\documentclass[a4paper,12pt]{article}
\usepackage[utf8]{inputenc}
\usepackage[english]{babel}
\usepackage{enumitem}
\usepackage{color}
\usepackage{amsmath}
\usepackage{amssymb}
\usepackage{amsthm}
\usepackage{amsfonts}
\usepackage{epsfig}
\usepackage{graphicx}
\usepackage{epstopdf}
\usepackage{multirow}
\usepackage{cite}
\usepackage{algorithmic}
\usepackage{algorithm}

\usepackage{caption}
\newtheorem{theorem}{Theorem}[section]

\newtheorem{lemma}{Lemma}[section]

\newtheorem{remark}{Remark}[section]

\theoremstyle{example}
\theoremstyle{definition}
\theoremstyle{Assumption}

\allowdisplaybreaks

\begin{document}
\title{\textbf{Inertial randomized Kaczmarz algorithms for solving coherent linear systems}}
\author{ \sc Songnian He$^{a,b}${\thanks{ Email: songnianhe@163.com}},\,\,
Ziting Wang$^{a}${\thanks{Email: wangziting0114@163.com}}\,\,\,and\,\,
Qiao-Li Dong$^{a,b}${\thanks{Corresponding author. Email: dongql@lsec.cc.ac.cn}}\\
\small Tianjin Key Laboratory for Advanced Signal Processing and College of Science, \\
\small Civil Aviation University of China, Tianjin 300300, China,
}
\date{}
\maketitle

\begin{abstract}
{In this paper, by regarding the two-subspace Kaczmarz method \cite{NW2013} as an alternated inertial randomized Kaczmarz algorithm we present a new convergence rate estimate which is shown to be better than that in \cite{NW2013} under a mild  condition. Furthermore, we accelerate the alternated inertial randomized Kaczmarz algorithm and  introduce a multi-step inertial randomized Kaczmarz algorithm which is  proved to have a faster convergence rate.
Numerical experiments support the theory results and illustrate that the multi-inertial randomized Kaczmarz algorithm significantly outperform  the two-subspace Kaczmarz method  in solving coherent linear systems.}
\end{abstract}

\noindent{\bf Keywords:} System of linear equations,  randomized Kaczmarz algorithm, inertial randomized Kaczmarz algorithm, inertial extrapolation, two-subspace Kaczmarz method.

\section{ Introduction }
Let $\mathbb{R}^{J}$ be the real vector space of $J$-dimensional column vectors with usual inner product $\langle\cdot,\cdot\rangle$
and norm $\|\cdot\|$,  respectively. We consider the following  system of linear equations:
\begin{equation}\label{Ax=b}
Ax=b,
\end{equation}
where $b\in \mathbb{R}^I$,  $A\in\mathbb{R}^{I\times J}$ and $x\in \mathbb{R}^J$.
Denote by $A_i,\,i=1,2, \cdots, I$,   the row vectors of $A$, then the linear equations in the system \eqref{Ax=b}  can be expressed as  $A_i x=b_i$, $i=1,2,\cdots, I$. As is known to all, associated with each equation $ A_ix=b_i$,   there is a hyperplane $H_{i}$
{defined  by}
\begin{equation*}
H_{i}=\{x{\in\mathbb{R}^{J}}\,|\, A_i x=b_i\},\,\,\,\,\,i=1,2,\cdots,I.
\end{equation*}
The orthogonal projection operator $P_{i}$ from $\mathbb{R}^{J}$ onto the hyperplane $H_{i}$ assigns to each $z\in\mathbb{R}^{J}$ the unique point in $H_i$ closest to $z$, denoted by $P_iz,$  (see (\ref{eq:2.2}) below for its closed expression).

Kaczmarz algorithm is {one of the most popular methods} to solve \eqref{Ax=b} (see, e.g., \cite{KM1937}, \cite[Chapter 8]{BCL14}), which can be expressed in terms of the operators $\{P_{i}\}_{i=1}^I$ as follows:
 \begin{equation*}
x^{k+1}=P_{i(k)}x^{k},
\end{equation*}
where $i(k)= (k \mod I) + 1$.
Because Kaczmarz algorithm uses only  single equation at each iteration step, it is also called a row-action \cite{{1},Censor81} or sequential method. If the system \eqref{Ax=b} is consistent, the sequence $\{x^k\}_{k=0}^\infty$ generated by Kaczmarz algorithm  converges to the solution closest to the initial value $x^{0}$.

The convergence speed of Kaczmarz algorithm heavily depends on the ordering of the equations. In selecting the equation ordering, the important thing is to avoid particularly bad ordering, in which case the hyperplanes $H_{i}$ and $H_{i+1}$ are nearly parallel  \cite{hm1993}, i.e.,  {$A_i$ and $A_{i+1}$ is highly correlated (its definition is given later). Furthermore}, it is very difficult to estimate the convergence rate of Kaczmarz algorithm.
\vskip 1mm

To overcome this difficulty, Strohmer and  Vershynin \cite{SV} proposed the following randomized Kaczmarz algorithm.

\begin{algorithm}{\rm \caption{Randomized Kaczmarz algorithm}}
\begin{algorithmic}
\label{Al:1.2}
\vskip 2mm
{
\STATE {\textbf{Input:} $A$, $b$ and $x^0$}
\STATE {\textbf{\quad $k\leftarrow k+1$}\\
\quad Select $i_k\in\{1,2,\ldots,I\}$\,\,\,$\{$\hbox{Randomly select a row of} $A$ with probability $\frac{\|A_{i_k}\|^2}{\|A\|_F^2}\}$\\
\quad  Set $x^{k+1}=P_{i_k}x^{k}$ \qquad\qquad\qquad\qquad\qquad\qquad\qquad\qquad\quad\,\,$\{$\hbox{Perform projection}$\}$
}
}
\end{algorithmic}
\end{algorithm}

\noindent
{Strohmer and  Vershynin proved that Algorithm 1 converges linearly in expectation when the linear system (\ref{Ax=b}) is consistent and $A$ is a full column rank matrix with $I\geq J$.}  Later, Ma and Needell et al. \cite{MNR2015} showed that whether  the matrix $A$ {satisfies} $I\geq J$ with full column  rank, or $I<J$ with  full row rank, Algorithm 1 converges linearly  in expectation to the minimum norm solution of the linear system (\ref{Ax=b}).
Furthermore, Gower and Richt\'{a}rik \cite{GR2015} proved that even if the full rank assumption imposed on the matrix $A$ is removed, the convergence result of Algorithm 1 still holds. Precisely, their convergence result is as follows.

\begin{theorem} {\rm\cite{GR2015}}
\label{th1.1}
Assume that the linear system \eqref{Ax=b} is consistent. Let the initial guess $x^0$ be selected in ${\rm span}\{A_1^\top, A_2^\top, \cdots, A_I^\top\}$ arbitrarily. Then the sequence $\{x^k\}_{k=0}^\infty$ generated by Algorithm {\rm1} converges linearly in expectation to the minimum norm solution $x^\dag$ for the linear system \eqref{Ax=b}, i.e.,
\begin{equation}
\label{RK-e}
\mathbb{E}\left [\|x^k-x^\dag\|^2\right ]\leq \left (1-\frac{\lambda_{\rm min}(A^\top A)}{\|A\|^2_F}\right )^k\|x^0-x^\dag\|^2, \,\,\,\,\forall k\geq 1,
\end{equation}
where $\lambda_{\rm min}(A^\top A)$   denotes  the minimum nonzero eigenvalue of $A^{\top}A$.
\end{theorem}

Highly correlated linear systems are often encountered in engineering practice, for example sampling in  geophysics which is the reconstruction of a bandlimited function from nonuniformly spaced  since it  can be physically challenging to take uniform samples. When the sampling is expressed as a system of linear equations, the sampling points form the rows of a matrix $A$ which are close together and therefore the corresponding rows will be highly correlated.

The quantity
$$
\delta(A)=\min_{i\neq j;\, 1\leq i,j\leq I}\left |\left\langle \frac{A_i}{\|A_i\|},\frac{A_j}{\|A_j\|}\right\rangle\right|.
$$
is generally used to measure {the coherence between the rows of the matrix $A$}.  The larger $\delta(A)$ {means  the higher  coherence  between row vectors of $A$}.  Hereafter, $\delta(A)$ is called the coherence constant of $A$, and the linear system \eqref{Ax=b} is called coherent if $\delta(A)>0$. It is easy to see that when $\delta(A)$ is large, the convergence speed of Kaczmarz algorithm (even randomized Kaczmarz algorithm) is relatively slow, regardless of the order of equations.
Therefore, it is of great significance to {present accelerated versions} of the randomized Kaczmarz algorithm which can effectively solve the coherent linear systems.
\vskip 1mm

In \cite{NW2013}, Needell and Ward introduced so called two-subspace Kaczmarz method  based on the idea of the orthogonalization, where the matrix  $A$ is standardized which  means that each of its rows has unit Euclidean norm before the  implementation of the method.

\begin{algorithm}{\rm \caption{Two-subspace Kaczmarz method}}
\label{TSS}
\begin{algorithmic}
\vskip 2mm
{
\STATE {\textbf{Input:} $A$, $b$ and $x^0$}
\STATE {\quad \textbf{$k\leftarrow k+1$}\\
\quad  Select $r,s\in\{1,2,\ldots,I\}$ $\,\,\{$\hbox{Select two distinct rows of} $A$ \hbox{uniformly at random}$\}$\\
\quad Set $\mu_k\leftarrow \langle a_r,a_s\rangle$ \qquad\qquad\qquad\qquad\qquad\qquad\qquad\qquad$\,\,\,\{\hbox{Compute correlation}\}$\\
\quad Set $y^k\leftarrow x^{k-1}+(b_s-\langle x^{k-1},a_s\rangle)a_s$\qquad\qquad\quad $\{\hbox{Perform intermediate projection}\}$\\
\quad Set $v_k\leftarrow\frac{\displaystyle a_r-\mu_ka_s}{\displaystyle 1-\mu_k^2}$
\qquad\,\, $\{$\hbox{Compute vector orthogonal to $a_s$ in direction of $a_r$}$\}$\\
\quad Set $\beta_k\leftarrow\frac{\displaystyle b_r-\mu_kb_s}{\displaystyle 1-\mu_k^2}$\qquad\qquad\qquad\qquad\, $\{$\hbox{Compute corresponding measurement}$\}$\\
\quad $x^k\leftarrow y^k+(\beta_k-\langle y^k,v_k\rangle)v_k$\qquad\qquad\qquad\qquad\qquad\qquad\qquad $\{$\hbox{Perform projection}$\}$
}
}
\end{algorithmic}
\end{algorithm}

\noindent
The  convergence rate estimate  was established  when $A$ is a full rank standardized matrix with  $I>J$ in \cite{NW2013}. In fact the full rank assumption can be removed as  Theorem \ref{th1.1} and the convergence results is given as follows.

\begin{theorem}
\label{th-tss}
Assume that the linear system \eqref{Ax=b} is consistent. Let the initial guess $x^0$ be selected in ${\rm span}\{A_1^\top, A_2^\top, \cdots, A_I^\top\}$ arbitrarily. Let $\{x^k\}$ be the sequence generated by the two-subspace Kaczmarz method. Then
\begin{equation}\label{TSS2}
\mathbb{E}\left [\|x^{k}-x^\dag\|^2\right]\leq \left[\left(1-\frac{\lambda_{\rm min}(A^\top A)}{\|A\|_F^2}\right)^2-\frac{\lambda_{\rm min}(A^\top A)D}{\|A\|_F^2}\right]^k\|x^0-x^\dag\|^2,
\end{equation}
where $D=\min\{\frac{\delta(A)^2(1-\delta(A))}{1+\delta(A)},\frac{\Delta(A)^2(1-\Delta(A))}{1+\Delta(A)}\}$ and $\Delta(A)=\max_{i\neq j;\, 1\leq i,j\leq I}\left|\left\langle \frac{A_i}{\|A_i\|},\frac{A_j}{\|A_j\|}\right\rangle\right|.$
\end{theorem}

\noindent
Comparing the estimates \eqref{RK-e} and (\ref{TSS2}), it is easy to see  that the two-subspace Kaczmarz method  improves the randomized Kaczmarz algorithm if any rows of $A$ are highly correlated.
\vskip 1mm

The inertial {extrapolation}  firstly introduced by Polyak \cite{Po1964} has been widely used to accelerate {the convergence speed of iterative algorithms} (see, e.g., \cite{Dong} and \cite{GOSB}) since the cost of each iteration stays basically unchanged.
To introduce the inertial extrapolation, we start by introducing the term {\bf ``basic algorithm"}. Consider the algorithmic operator ${\bf\Psi}: \mathbb{R}^J\rightarrow\mathbb{R}^I $  working iteratively by
\begin{equation}
\label{BA}
x^{k+1}={\bf\Psi}(x^k),\,\,\,\,\,\,\forall k\geq 0,
\end{equation}
which is denoted as the basic algorithm.
The {\bf inertial version} of the basic algorithm \eqref{BA} is defined as follows:
\begin{equation}
\label{iBA}
x^{k+1}={\bf\Psi}(x^k+\alpha_k(x^k-x^{k-1})),\,\,\,\,\,\,\forall k\geq 1,
\end{equation}
where $\alpha_k(x^k-x^{k-1})$ is referred to as the inertial term and $\alpha_k$ as the inertial parameter.  The main feature of inertial type methods is that the next iterate is defined by making use of the previous two iterates.
The inertial parameter sequence $\{\alpha_k\}_{k=0}^\infty$ is key to inertial type algorithms  and a
careful choice   is known to accelerate the  convergence rate of the objective function from $O(1/k)$ to $O(1/k^2)$ or $o(1/k^2)$  for a large class of {first-order}
algorithms of the convex optimization (see \cite{Atto16,Be09,Nest83} for details). Since its inception,  inertial type method has received a great deal of attention of
many authors, who improved it in various ways.
Mu and Peng \cite{Mu-Peng} firstly introduced alternated inertial methods applying inertia every other iteration, which was extended to solve the composite convex optimization problem \cite{Iutzeler} and variational inequalities \cite{Shehu}. Recently multi-step inertial methods \cite{Liang} were proposed based on the idea of using previous more than two  iterates to accelerate the convergence rate.
\vskip 1mm

The purpose of this paper is three folds. Firstly, we give a new convergence rate estimate of the two-subspace Kaczmarz method by rewriting it as an alternated inertial randomized Kaczmarz algorithm and show that it is better  than that in \cite{NW2013}  under a mild condition.
Secondly, we introduce a multi-step inertial randomized Kaczmarz algorithm by combining the inertial extrapolation and the randomized Kaczmarz algorithm and its convergence rate is proved to be faster than the two-subspace Kaczmarz method. Finally,  numerical experiments support the theory results and illustrate that the proposed algorithm outperform the two-subspace Kaczmarz method  in solving coherent linear systems.

The paper is structured as follows. In the next section, we introduce some fundamental  concepts and tools, which are  needed in the subsequent sections. In Section 3, we rewrite the two-subspace Kaczmarz method as  the alternated inertial randomized Kaczmarz algorithm and present a new convergence rate. In Section 4,  we propose the multi-step inertial randomized Kaczmarz algorithm as a further extension of the alternated inertial randomized Kaczmarz algorithm. Numerical experiments are presented in Section 5 to show the efficiency and advantage of our proposed algorithm.

\section{Preliminaries}

In this section, some fundamental tools are listed  that will be used in the convergence analysis of the proposed methods.

We will  use the following notations:
\begin{enumerate}
\item [$\bullet$] $\mathcal{I}=\{1,2,\cdots,I\}$ and $\Lambda=\{(j,i)\,|\,j,i\in \mathcal{I}\,{\rm with}\,j\neq i\}$ denote two index sets, respectively.
\item [$\bullet$]$\lambda_{\min}(A^\top A)$ denotes the minimum nonzero eigenvalue of $A^{\top}A$.
\item [$\bullet$]$\|A\|_F$ denotes the Frobenius norm of $A$.
\item [$\bullet$] $\alpha_F (A)=\max_{i\in\mathcal{I}}\{\|A\|^2_F-\|A_i\|^2\}$ and $\Upsilon=\sum_{(j,i)\in \Lambda}\|A_j\|^2\|A_i\|^2$ denote two constants, respectively.
\item [$\bullet$] $x^\dag$ denotes the minimum norm solution of the problem (\ref{Ax=b}).
\item [$\bullet$] $\mathbb{E}\left[\cdot\right]$ and $\mathbb{E}\left[\cdot\,|\,\cdot\right]$ denote the full expectation and the conditional expectation of a random variable, respectively.
\end{enumerate}

For a given  nonempty closed convex subset $C$ of $\mathbb{R}^{J}$, the projection operator onto $C$
is the mapping from $\mathbb{R}^{J}$ to $C$ that assigns to each $x\in \mathbb{R}^{J}$
the unique point in $C$, denoted $P_Cx$,  closest to $x$ that is,
$$P_Cx=\arg\min_{z\in C}\|z-x\|^2.$$

A subset $H$ of $\mathbb{R}^{J}$ is called a hyperplane if $H$ is  of the form
\begin{equation}\label{eq:2.1}
H:=\{x\in \mathbb{R}^{J} \mid \langle x,v \rangle=d\},
\end{equation}
where $v$ is a fixed nonzero vector in $\mathbb{R}^{J}$ and $d$ is a number in $\mathbb{R}$. Obviously, a hyperplane is closed and convex,
and we have the following closed-form formula for projections onto  a hyperplane $H$ (see \cite[Chapter 25]{BC10} or \cite{hyd}).
\begin{lemma}
Let $H$ be a hyperplane of the form \eqref{eq:2.1}. Then, for any $u\in \mathbb{R}^{J}$, we have
\begin{equation}\label{eq:2.2}
P_{H}u=u-\frac{\langle u,v \rangle-d}{\|v\|^{2}}v.
\end{equation}
\end{lemma}
\noindent
In what follows, we denote by $P_i$ the projection operator onto $H_{i}=\{x{\in\mathbb{R}^{J}}\,|\, A_i x=b_i\},i=1,2,\cdots,I.$
\vskip 1mm

The following property is a well-known characterization of projections.
\begin{lemma} \label{le:2.2} {\rm(\cite[Section 3]{GR84})}
Let $C$ be a closed convex subset of $\mathbb{R}^J.$ Given  $x\in \mathbb{R}^J$ and  $z\in C$. Then $z=P_{C}x$ if and only if
\begin{equation}\label{eq:2.3}
\langle x-z,y-z \rangle\leq0,\quad \forall\,y\in C.
\end{equation}
Moreover, if $C$ is a hyperplane, then $z=P_{C}x$ if and only if
\begin{equation}\label{eq:2.4}
\langle x-z,y-z \rangle=0,\quad\forall\,y\in C.
\end{equation}

\end{lemma}

\begin{lemma} \label{le:2.3} \cite{GR2015}
Let $A\in \mathbb{R}^{I\times J}$ be a matrix with row vectors $A_1, A_2, \cdots, A_I$. Then the inequality
\begin{equation}\label{eq:2.5}
\|Ax\|^2\geq \lambda_{\min}(A^{\top}A)\|x\|^2
\end{equation}
holds for all  $x\in {\rm span}\{A_1^\top, A_2^\top, \cdots, A_I^\top\}$.\emph{}
\end{lemma}

Throughout the rest of this paper, all discussions are based on the following two basic assumptions:
\begin{enumerate}
  \item [({H1})] The system of linear equations \eqref{Ax=b} is consistent, that is, its solution set, {denoted by $S$,} is nonempty.
  \item [({H2})] Any two equations in \eqref{Ax=b} are different, that is, the same equation is not allowed to appear twice in \eqref{Ax=b}.
\end{enumerate}

\section{Alternated inertial randomized  Kaczmarz algorithm}

{
In this section,  we will give a better convergence rate estimate of the two-subspace Kaczmarz method than (\ref{TSS2}) by regarding it as an alternated inertial randomized  Kaczmarz algorithm.
\vskip 2mm

To this end, we first introduce an alternated inertial randomized  Kaczmarz algorithm as follows.
\begin{algorithm}{\rm \caption{Alternated inertial randomized Kaczmarz algorithm}}
\label{Al:3.1}
\begin{algorithmic}
\vskip 2mm
{
\STATE {\textbf{Input:} $A$ and $b$}
\STATE {\textbf{Step 0:} Take $x^0\in {\rm span}\{A_1^\top, A_2^\top, \cdots, A_I^\top\}$  arbitrarily and set $k=1$.\\
\STATE {\textbf{Step 1:} Select an index $(j_k, i_k)\in \Lambda$ with probability $p_{(j_k,i_k)}=\frac{\|A_{j_k}\|^2\|A_{i_k}\|^2}{\Upsilon}$.}\\
\STATE {\textbf{Step 2:}  Set
\begin{equation}\label{eq:3.1}
y^k=P_{j_k}x^{k-1},
\end{equation}
\qquad\qquad  Set
\begin{equation}\label{eq:3.2}
\beta_k=\frac{(A_{i_k}y^k-b_{i_k})\langle A_{j_k}, A_{i_k}\rangle}{\|A_{j_k}\|^2\|A_{i_k}\|^2-\langle A_{j_k}, A_{i_k}\rangle^2},
\end{equation}
\qquad\qquad  Calculate
\begin{equation}\label{eq:3.3}
x^k=P_{i_k}(y^k+\beta_k A_{j_k}^\top).
\end{equation}
}
\STATE {\textbf{Step 3:}
Set $k:=k+1$ and return to Step 1.
}
}
}
\end{algorithmic}
\end{algorithm}

\vskip 100mm
\begin{remark}
\rm
We give three important characteristics of Algorithm 3.

(i)   $\beta_k$ given by (\ref{eq:3.2}) is well-defined for all $k\geq 0$. Indeed, since $j_k\neq i_k$ can be guaranteed by Algorithm 3, it is easy to see that $\|A_{j_k}\|^2\|A_{i_k}\|^2-\langle A_{j_k}, A_{i_k}\rangle^2\neq 0$ under the  assumption (H2).

(ii) Algorithm 3 is essentially an alternated inertial algorithm. In fact, from (\ref{eq:3.1}) and (\ref{eq:2.2}), we get
\begin{equation} \label{e3.2}
y^k-x^{k-1}=-\frac{A_{j_k}x^{k-1}-b_{j_k}}{\|A_{j_k}\|^2}A_{j_k}^\top.
\end{equation}
Obviously,  in the case where $y^k\neq x^{k-1}$, we have
\begin{equation}\label{new-add1}
y^k+\beta_k A_{j_k}^\top=y^k+\beta^\prime_k (y^k-x^{k-1}),
\end{equation}
where
\begin{equation} \label{e3.5}
\beta^\prime_k=-\frac{(A_{i_k}y^k-b_{i_k})\|A_{j_k}\|^2\langle A_{j_k}, A_{i_k}\rangle}{(A_{j_k} x^{k-1}-b_{j_k})\left[\|A_{j_k}\|^2\|A_{i_k}\|^2-\langle A_{j_k}, A_{i_k}\rangle^2\right]}.
\end{equation}
Consequently, (\ref{eq:3.3}) can be rewritten as the form
\begin{equation}\label{new-add2}
x^k=P_{i_k}(y^k+\beta^\prime_k (y^k-x^{k-1})).
\end{equation}
The formula (\ref{new-add2}) means that there is an inertial extrapolation in $x^k$. Since there is no inertial extrapolation in $y^k$, Algorithm 3 is an alternated inertial algorithm.

(iii) When $A$ is standardized, it is easy to verify that Algorithm 3 equals to the two-subspace Kaczmarz method. The reason that we prefer the form  of Algorithm 3 instead of the two-subspace Kaczmarz method in this section is that the former is convenient for us to make full use of the properties of the projection operator to analyze the convergence rate more precisely, which can be seen from Lemma \ref{le:3.3} and its proof  below.
\end{remark}
}

We now give the following fundamental result, which is crucial  to estimate the convergence rate of Algorithm 3.

\begin{lemma} \label{le:3.3}
Let $\{x^k\}_{k=0}^\infty$ be the  sequence generated by Algorithm {\rm3}. Then for any $x^* \in S$ and $k \geq 1$, the following statements hold:
\begin{enumerate}
  \item [{\rm(i)}]  $  \|y^k-x^*\|^2=\|x^{k-1}-x^*\|^2-\frac{\displaystyle|A_{j_k}x^{k-1}-b_{j_k}|^2}
      {\displaystyle\|A_{j_k}\|^2};$

  \item [{\rm(ii)}] $\|y^k-x^{k-1}\|=\frac{\displaystyle|A_{j_k}x^{k-1}-b_{j_k}|}{\displaystyle\|A_{j_k}\|}$;
  \item [{\rm(iii)}] $\left\langle A^\top_{j_k},  y^k-x^*\right\rangle =0$;
  \item [{\rm(iv)}] The inertial parameter sequence $\{\beta_k\}_{k=1}^\infty$ in \eqref{eq:3.2} is optimal in the sense that $\beta_k$ makes $\|x^k-x^*\|^2$ obtain its minimum value;
  \item [{\rm(v)}]
\begin{equation*}
\|x^k-x^*\|^2=\|y^k-x^*\|^2-\frac{\displaystyle|A_{i_k}y^k-b_{i_k}|^2}{\displaystyle\|A_{i_k}\|^2}\frac{\displaystyle1}
{\displaystyle1-\left\langle \frac{\displaystyle A_{j_k}}{\displaystyle\|A_{j_k}\|},\frac{\displaystyle A_{i_k}}{\displaystyle\|A_{i_k}\|}\right\rangle^2};
\end{equation*}

 \item [(vi)]

 \begin{equation*}
\|x^k-y^k\|^2=\frac{\displaystyle|A_{i_k}y^k-b_{i_k}|^2}{\displaystyle\|A_{i_k}\|^2}\frac{\displaystyle1}
{\displaystyle1-\left\langle \frac{\displaystyle A_{j_k}}{\displaystyle\|A_{j_k}\|},\frac{\displaystyle A_{i_k}}{\displaystyle\|A_{i_k}\|}\right\rangle^2}.
\end{equation*}

\end{enumerate}
\end{lemma}

\begin{proof} (i) Noting the fact that $A_{j_k}x^*=b_{j_k}$, we get
$$A_{j_k}(x^{k-1}-x^*)=A_{j_k}x^{k-1}-A_{j_k}x^*=A_{j_k}x^{k-1}-b_{j_k},$$
and consequently, this together with (\ref{eq:3.1}) and (\ref{eq:2.2}) implies that

\begin{equation}\label{eq:3.4}
\aligned
&\|y^k-x^*\|^2\\
=&\|P_{j_k}x^{k-1}-x^*\|^2\\
=&\|x^{k-1}-x^*-\frac{A_{j_k}x^{k-1}-b_{j_k}}{\|A_{j_k}\|^2}A_{j_k}^\top\|^2\\
=&\|x^{k-1}-x^*\|^2-2\left\langle x^{k-1}-x^*,\frac{A_{j_k}x^{k-1}-b_{j_k}}{\|A_{j_k}\|^2}A_{j_k}^\top\right\rangle +\frac{| A_{j_k}x^{k-1}-b_{j_k}|^2}{\|A_{j_k}\|^2}\\
=&\|x^{k-1}-x^*\|^2-\frac{|A_{j_k}x^{k-1}-b_{j_k}|^2}{\|A_{j_k}\|^2}.
\endaligned
\end{equation}

(ii) From (\ref{e3.2}),  we have
$$\|y^k-x^{k-1}\|=\frac{|A_{j_k}x^{k-1}-b_{j_k}|}{\|A_{j_k}\|}.$$

(iii) Noting $A_{j_k}x^*=b_{j_k}$, this conclusion is a direct consequence of  (\ref{e3.2}).

(iv) Obviously,  in order to make Algorithm 3 obtain the optimal convergence rate, for the current $y^k$, $\beta_k$ should be chosen  such that $\|x^k-x^*\|^{2}$ reaches its minimum value. To show that the inertial parameter sequence $\{\beta_k\}_{k=1}^\infty$ in (\ref{eq:3.2}) is optimal,  it suffices to verify that
$$ \beta_k=\arg \min\limits_{\beta \in \mathbb{R}}\|P_{i_k}(y^k+\beta A_{j_k}^\top)-x^*\|^2$$
holds for all $k\geq 1$.

Similar to the derivation of (\ref{eq:3.4}), we get
\begin{align*}
&\|P_{i_k}(y^k+\beta A_{j_k}^\top)-x^*\|^2\\
=&\|y^k-x^*+\beta A_{j_k}^\top\|^2-\frac{|A_{i_k}(y^k+\beta A_{j_k}^\top) -b_{i_k}|^2}{\|A_{i_k}\|^2}.
\end{align*}
Using (iii), one has
\begin{align*}
\|y^k-x^*+\beta A_{j_k}^\top\|^2=&\|y^k-x^*\|^2+\beta^2\|A_{j_k}\|^2+2\beta \langle y^k-x^*,A_{j_k}^\top\rangle\\
=&\|y^k-x^*\|^2+\beta^2\|A_{j_k}\|^2.
\end{align*}
On the other hand,
\begin{align*}
|A_{i_k}(y^k+\beta A_{j_k}^\top) -b_{i_k}|^2=&|A_{i_k}y^k-b_{i_k}|^2+\beta^2\langle A_{j_k}, A_{i_k}\rangle^2\\
&+2\beta\langle A_{j_k}, A_{i_k}\rangle (A_{i_k}y^k-b_{i_k}).
\end{align*}
Combining  above three equalities, we have
\begin{equation}\label{eq:3.5}
\aligned
& \|P_{i_k}(y^k+\beta A_{j_k}^\top)-x^*\|^2\\
=&\|y^k-x^*\|^2-\frac{|A_{i_k}y^k-b_{i_k}|^2}{\|A_{i_k}\|^2}\\
&+\beta^2\left (\|A_{j_k}\|^2-\frac{\langle A_{j_k}, A_{i_k}\rangle^2}{\|A_{i_k}\|^2}\right )-2\beta \frac{\langle A_{j_k}, A_{i_k}\rangle(A_{i_k}y^k-b_{i_k})}{\|A_{i_k}\|^2}.
\endaligned
\end{equation}
Set
\begin{equation}\label{eq:3.6}
\aligned
h(\beta):=&\beta^2\left (\|A_{j_k}\|^2-\frac{\langle A_{j_k}, A_{i_k}\rangle^2}{\|A_{i_k}\|^2}\right )\\
&-2\beta \frac{\langle A_{j_k}, A_{i_k}\rangle(A_{i_k}y^k-b_{i_k})}{\|A_{i_k}\|^2},\,\,\,\forall\beta\in\mathbb{R}.
\endaligned
\end{equation}
For the current iterate $y^k$, noting that
$$\|y^k-x^*\|^2-\frac{|A_{i_k}y^k-b_{i_k}|^2}{\|A_{i_k}\|^2}$$
is a constant, it is easy to see from (\ref{eq:3.5}) and (\ref{eq:3.6}) that, $\|P_{i_k}(y^k+\beta A_{j_k}^\top)-x^*\|^2$ reaching its minimum value is equivalent to the function $h(\beta)$ reaching its minimum value. Note that $h(\beta)$ is a quadratic function of $\beta$, therefore, minimizing $h(\beta)$
yields that the $k$th optimal inertial parameter is given by  (\ref{eq:3.2}).

(v) The conclusion can be obtained directly by setting $\beta=\beta_k$ in (\ref{eq:3.5}).

(vi) Using (\ref{eq:3.2}) and (\ref{eq:3.3}), we have
\begin{equation}\label{eq:3.9}
x^k=y^k+\frac{(A_{i_k}y^k-b_{i_k})\left[\langle A_{j_k}, A_{i_k}\rangle A_{j_k}^\top-\|A_{j_k}\|^2A_{i_k}^\top\right] }{\|A_{j_k}\|^2\|A_{i_k}\|^2-\langle A_{j_k}, A_{i_k}\rangle^2}.
\end{equation}
The desired conclusion follows from (\ref{eq:3.9}).
    \end{proof}

\begin{remark}
\rm
(i) Lemma \ref{le:3.3} (iv) provides a reference for us to find $\beta_k$, whose idea is  just the  selection strategy of inertial parameters in some inertial type algorithms.

(ii)
The accelerating convergence effect of  Algorithm 3 can be clearly shown by  (v) and (vi) of Lemma \ref{le:3.3}. In fact, the iterative sequence generated by  Algorithm 1 only holds

$$
\|x^k-x^*\|^2= \|y^k-x^*\|^2-\frac{|A_{i_k}y^k-b_{i_k}|^2}{\|A_{i_k}\|^2},
$$
and
$$\|x^k-y^k\|^2=\frac{|A_{i_k}y^k-b_{i_k}|^2}{\|A_{i_k}\|^2}.$$
While, due to the inertia extrapolation, the iterative sequence generated by Algorithm 3 has  smaller $\|x^k-x^*\|^2 $ and bigger $\|x^k-y^k\|^2$. Indeed, from (v) and (vi) of Lemma \ref{le:3.3}, we get

        $$\|x^k-x^*\|^2\leq \|y^k-x^*\|^2-\frac{|A_{i_k}y^k-b_{i_k}|^2}{\|A_{i_k}\|^2}\frac{1}{1-\delta(A)^2},$$
and
$$\|x^k-y^k\|^2\geq\frac{|A_{i_k}y^k-b_{i_k}|^2}{\|A_{i_k}\|^2}\frac{1}{1-\delta(A)^2}.$$
Obviously, the larger $\delta(A)$, the better the acceleration effect of Algorithm 3.

        \end{remark}

Finally we analyze the  convergence  rate   of Algorithm 3,  i.e, the two-subspace Kaczmarz method  as follows.

\begin{theorem}\label{Th1}
    Choose $x^0\in {\rm span}\{A_1^\top, A_2^\top, \cdots, A_I^\top\}$ arbitrarily and let $\left\{x^k\right\}_{k=0}^{\infty}$ be a sequence generated by  Algorithm {\rm3}. Then $\left\{x^k\right\}_{k=0}^{\infty}$ converges linearly in expectation  to the minimum norm solution $x^\dag$. Precisely, for all $k\geq 1$, there holds the convergence rate estimate:

 \begin{equation} \label{e24}
\mathbb{E}\left [\|x^{k}-x^\dag\|^2\right]\leq \left (1-\frac{\lambda_{\min}(A^\top A)}{\alpha_F (A)(1-\delta^2(A))}\right)^k\left(1-\frac{\lambda_{\min}(A^\top A)}{\|A\|_F^2}\right)^k\|x^{0}-x^\dag\|^2.
\end{equation}
     \end{theorem}

\begin{proof} Without losing generality, we assume that the system of equations (\ref{Ax=b}) has been standardized, that is, $\|A_i\|^2=1$ holds  for each $i\in \mathcal{I}$. In this case, it is very easy to see that $\|A\|_F^2=I$, $\alpha_F(A)=I-1$, and $\Upsilon=I(I-1)$ hold. Consequently,  Step 1 of Algorithm 3 is to select $(j_k,i_k)\in \Lambda$ uniformly at random, i.e.,  with probability $\frac{1}{I(I-1)}$.

 Note that $A_ix^\dag=b_i$ holds for all $i\in \mathcal{I}$, we have  from (i) and (v) of Lemma \ref{le:3.3} that
    \begin{equation}\label{eq:3.11}
    \|y^k-x^\dag\|^2=\|x^{k-1}-x^\dag\|^2-\displaystyle|A_{j_k}(x^{k-1}-x^\dag)|^2
    \end{equation}
 and
 \begin{equation}\label{r2}
\|x^{k}-x^\dag\|^2
\leq\|y^k-x^\dag\|^2-\frac{\displaystyle|A_{i_{k}}(y^k-x^\dag)|^2}{(1-\delta^2(A))}
\end{equation}
hold for all $k\geq 1$, respectively. By (\ref{r2}) and the definition of  conditional expectation, we have

\begin{equation}\label{r1y}
 \aligned
&\mathbb{E}\left[\|x^k-x^\dag\|^2\,|\,x^{k-1}\right] \\
\leq& \frac{1}{I(I-1)}\sum_{(j_k,i_k)\in \Lambda}\left\{\|y^k-x^\dag\|^2-\frac{\displaystyle|A_{i_{k}}(y^k-x^\dag)|^2}{(1-\delta^2(A))}\right\}\\
 =&\frac{1}{I(I-1)}\sum_{j_k\in \mathcal{I}}\sum_{i_k\in\mathcal{I}-\{j_k\}}\left\{\|y^k-x^\dag\|^2-\frac{\displaystyle|A_{i_{k}}(y^k-x^\dag)|^2}{(1-\delta^2(A))}\right\}\\
 =&\frac{1}{I}\sum_{j_k\in \mathcal{I}}\|y^k-x^\dag\|^2-\frac{1}{I(I-1)(1-\delta^2(A))}\sum_{j_k\in \mathcal{I}}\sum_{i_k\in\mathcal{I}-\{j_k\}}\displaystyle|A_{i_{k}}(y^k-x^\dag)|^2.
 \endaligned
 \end{equation}
Noting that $y^{k}-x^\dag\in {\rm span}\{A_1^\top, A_2^\top, \cdots, A_I^\top\}$ and   $A_{j_k}(y^k-x^\dag)=0$ due to $y^k=P_{j_k}x^{k-1} $, it follows from Lemma \ref{le:2.3} that
\begin{equation}\label{r4}
\aligned
\sum_{i_k\in\mathcal{I}-\{j_k\}}\displaystyle|A_{i_{k}}(y^k-x^\dag)|^2=&\sum_{i_k\in\mathcal{I}}\displaystyle|A_{i_{k}}(y^k-x^\dag)|^2\\
=&\|A(y^k-x^\dag)\|^2
\geq \lambda_{\min}(A^\top A)\|y^k-x^\dag\|^2.
\endaligned
\end{equation}
Combining (\ref{r1y}) and (\ref{r4}), we have
\begin{equation}\label{r2y}
\mathbb{E}\left[\|x^k-x^\dag\|^2\,|\,x^{k-1}\right]\leq \left(1-\frac{\lambda_{\min}(A^\top A)}{(I-1)(1-\delta^2(A))}\right)\frac{1}{I}\sum_{j_k\in \mathcal{I}}\|y^k-x^\dag\|^2.
\end{equation}

Noting $x^{k-1}-x^\dag\in {\rm span}\{A_1^\top, A_2^\top, \cdots, A_I^\top\}$,   we have from (\ref{eq:3.11}) and
 Lemma \ref{le:2.3} that
 \begin{equation}\label{r1}
 \aligned
\frac{1}{I}\sum_{j_k\in \mathcal{I}}\|y^k-x^\dag\|^2 =& \|x^{k-1}-x^\dag\|^2-\frac{1}{I}\sum_{j_k\in \mathcal{I}}|A_{j_k}(x^{k-1}-x^\dag)|^2\\
 =&\|x^{k-1}-x^\dag\|^2-\frac{1}{I}\|A(x^{k-1}-x^\dag)\|^2\\
 \leq &\left (1-\frac{\lambda_{\min}(A^\top A)}{I}\right)\|x^{k-1}-x^\dag\|^2.
 \endaligned
 \end{equation}
Substituting (\ref{r1}) into (\ref{r2y}), we have
\begin{equation}\label{r3y}
\mathbb{E}\left[\|x^k-x^\dag\|^2\,|\,x^{k-1}\right]\leq \left(1-\frac{\lambda_{\min}(A^\top A)}{(I-1)(1-\delta^2(A))}\right)\left (1-\frac{\lambda_{\min}(A^\top A)}{I}\right)\|x^{k-1}-x^\dag\|^2.
\end{equation}
Hence, taking the full expectation at both sides of (\ref{r3y}), we obtain
\begin{equation}\label{r6}
\mathbb{E}\left[\|x^k-x^\dag\|^2\right]\leq \left(1-\frac{\lambda_{\min}(A^\top A)}{(I-1)(1-\delta^2(A))}\right)\left(1-\frac{\lambda_{\min}(A^\top A)}{I}\right)\mathbb{E}\left[\|x^{k-1}-x^\dag\|^2\right].
\end{equation}
Noting $\|A\|_F^2=I$ and $\alpha_F(A)=I-1$, (\ref{e24}) follows by applying  mathematical induction to (\ref{r6}).    This completes the proof.
\end{proof}

Although Algorithm 3 is equivalent to the two-subspace Kaczmarz method when $A$ is standardized, the proof of Theorem \ref{Th1}  is completely different from that of Theorem \ref{th-tss}, and the former is simpler. Furthermore, the convergence rate estimate (\ref{e24}) and (\ref{TSS2}) are  different  and  obviously, (\ref{e24}) is better than (\ref{TSS2}) in at least   two aspects:
 \begin{enumerate}
  \item [(i)]  (\ref{e24}) is only related to $\delta(A)$,  while (\ref{TSS2}) is related to both $\delta(A)$ and $\Delta(A)$;
  \item [(ii)] When $\delta(A)=0$, the acceleration effect of the two-subspace Kaczmarz method can not be reflected by (\ref{TSS2}) due to $D=0$,  but (\ref{e24})  clearly shows  the acceleration effect of Algorithm 3 since $\alpha_F(A)<\|A\|_F^2$.
 \end{enumerate}
Further, we can prove  that (\ref{e24}) is better than (\ref{TSS2}) under a mild condition.

\begin{theorem}
If $2\leq r\leq I\leq 1+\frac{9\left(1+\sqrt{1-\frac{8}{9r}}\right)^2r^2(r-1)}{4},$ then {\rm(\ref{e24})} is better than {\rm(\ref{TSS2})}, where $r$ is the rank of $A$.
\end{theorem}
\begin{proof}
For convenience, we assume that the system of equations (\ref{Ax=b}) has been standardized, i.e., $\|A_i\|^2=1$ for all  $i\in\mathcal{I}$. Also, $\lambda_{\rm min}(A^\top A)$ and $\delta(A)$ are simply expressed as $\lambda_{\rm min}$ and $\delta$, respectively.  Hence (\ref{e24}) and (\ref{TSS2}) become
 \begin{equation} \label{new-e24}
\mathbb{E}\left [\|x^{k}-x^\dag\|^2\right]\leq \left (1-\frac{\lambda_{\min}}{(I-1)(1-\delta^2)}\right)^k\left(1-\frac{\lambda_{\min}}{I}\right)^k\|x^{0}-x^\dag\|^2.
\end{equation}
and
\begin{equation}\label{new-TSS}
\mathbb{E}\left [\|x^{k}-x^\dag\|^2\right]\leq \left[\left(1-\frac{\lambda_{\rm min}}{I}\right)^2-\frac{\lambda_{\rm min}D}{I}\right]^k\|x^0-x^\dag\|^2,
\end{equation}
respectively. Note that
$$\left (1-\frac{\lambda_{\min}}{(I-1)(1-\delta^2)}\right)\left(1-\frac{\lambda_{\min}}{I}\right)=\left(1-\frac{\lambda_{\rm min}}{I}\right)^2-\frac{\lambda_{\rm min}}{I}\left(1-\frac{\lambda_{\rm min}}{I}\right)\frac{1+(I-1)\delta^2}{(I-1)(1-\delta^2)},$$ to prove the desired conclusion, from (\ref{new-e24}) and (\ref{new-TSS}), we need to show that
\begin{equation}\label{TSS-1}
\left(1-\frac{\lambda_{\rm min}}{I}\right)\frac{1+(I-1)\delta^2}{(I-1)(1-\delta^2)}\geq D
\end{equation}
holds for all $\delta\in[0,1].$ Since $r\lambda_{\rm min}\leq I$ and $D\leq \frac{\delta^2(1-\delta)}{1+\delta}$, to show (\ref{TSS-1}), it suffices to verify

$$\left(1-\frac{1}{r}\right)\frac{1+(I-1)\delta^2}{(I-1)(1-\delta^2)}\geq \frac{\delta^2(1-\delta)}{1+\delta},\,\,\forall \delta\in[0,1],$$
that is,
\begin{equation}\label{TSS-2}
\left(1-\frac{1}{r}\right)\left(\frac{1}{I-1}+\delta^2\right)\geq \delta^2(1-\delta)^2,\,\,\forall \delta\in[0,1].
\end{equation}

Define a real function:
\begin{equation}\label{TSS-3}
f(\delta)=\left(1-\frac{1}{r}\right)\left(\frac{1}{I-1}+\delta^2\right)-\delta^2(1-\delta)^2,\,\,\forall \delta\in[0,1].
\end{equation}
It is easy to verify that $f$ obtains its minimum value at $\hat{\delta}=\frac{2}{3r(1+\sqrt{1-\frac{8}{9r}})}$ and this implies  that proving (\ref{TSS-2}) is equivalent to proving $f(\hat{\delta})\geq 0$. To prove $f(\hat{\delta})\geq 0$, from (\ref{TSS-3}), it suffices to verify that
\begin{equation}\label{TSS-4}
\left(1-\frac{1}{r}\right)\left(\frac{1}{I-1}+\hat{\delta}^2\right)\geq \hat{\delta}^2
\end{equation}
holds. Indeed, it is very easy to verify that (\ref{TSS-4}) is equivalent to the condition $2\leq r\leq I\leq 1+\frac{9(1+\sqrt{1-\frac{8}{9r}})^2r^2(r-1)}{4}.$ This completes the proof.
\end{proof}

\begin{remark}
\rm
A great deal of matrices satisfy the condition $2\leq r\leq I\leq 1+\frac{9(1+\sqrt{1-\frac{8}{9r}})^2r^2(r-1)}{4}$, for example, the row full rank matrix, and column full rank matrix with $I\le 1+\frac{9(1+\sqrt{1-\frac{8}{9J}})^2J^2(J-1)}{4}$.

\end{remark}
\vskip -2mm

\section{Multi-step inertial randomized Kaczmarz algorithm}

The alternated inertial randomized Kaczmarz algorithm is computationally efficient  and simple to implement, but this method still could be accelerated. In this section,  we accelerate the alternated inertial randomized Kaczmarz algorithm based on the idea in multi-step inertial algorithms and thus name it as  the multi-step inertial randomized Kaczmarz algorithm.

\begin{algorithm}[!h]{\rm \caption{Multi-step inertial randomized Kaczmarz algorithm}}
\label{Al:4.1}
\begin{algorithmic}
\vskip 2mm
{
\STATE {\textbf{Input:}\, $A$ and $b$}
\STATE {\textbf{Step 0:} Take $x^0\in {\rm span}\{A_1^\top, A_2^\top, \cdots, A_I^\top\}$  arbitrarily and set $w^0=x^0$.\\
\qquad\qquad Select an index $i_0\in\{1,2,\cdots,I\}$ with probability $p_{i_0}=\frac{\|A_{i_0}\|^2}{\|A\|^2_{F}}$.\\
\qquad\qquad Calculate
\begin{equation}\label{eq:4.1}
x^1=P_{i_0}w^0.
\end{equation}
\qquad\qquad Set $k=1$.

\STATE {\textbf{Step 1:} Select an index $i_k\in \mathcal{I}-\{i_{k-1}\}$    with probability  $p_{i_k}=\frac{\|A_{i_k}\|^2}{\|A\|^2_{F}-\|A_{i_{k-1}}\|^2}$,}\\
\STATE {\textbf{Step 2:} Set
\begin{equation}\label{eq:4.2}
\gamma_k=\frac{(A_{i_k}x^k-b_{i_k})\langle A_{i_{k-1}}, A_{i_k}\rangle }{\|A_{i_{k-1}}\|^2\|A_{i_k}\|^2-\langle A_{i_{k-1}}, A_{i_k}\rangle^2}.
\end{equation}
\qquad\qquad Set
\begin{equation}\label{eq:4.3}
w^{k}=x^{k}+\gamma_kA_{i_{k-1}}^\top.
\end{equation} \\
\qquad\qquad Calculate
\begin{equation}\label{eq:4.4}
x^{k+1}=P_{i_k}w^k.
\end{equation}
}
\STATE {\textbf{Step 3:}
Set $k:=k+1$ and return to Step 1.
}
}
}
\end{algorithmic}
\end{algorithm}

\begin{remark}
\rm 

There are two illustrations to Algorithm 4.

(i) Algorithm 4 is essentially a  multi-step inertial algorithm. For each $k\geq 1$, noting $x^k=P_{i_{k-1}}w^{k-1}$, we get
\begin{equation}\label{add1}
x^k-w^{k-1}=-\frac{A_{i_{k-1}}w^{k-1}-b_{i_{k-1}}}{\|A_{i_{k-1}}\|^2}A_{i_{k-1}}^\top.
\end{equation}
It is easy to see from (\ref{eq:4.3}) and (\ref{add1}) that, in the case $x^k\neq w^{k-1}$,  (\ref{eq:4.3}) can be rewritten as the form:
\begin{equation}\label{4new1}
w^k=x^k+\gamma_k^\prime(x^k-w^{k-1}),
\end{equation}
where
\begin{equation}\label{e4.1}
\gamma_k^\prime=-\frac{(A_{i_k}x^k-b_{i_k})\langle A_{i_{k-1}}, A_{i_k}\rangle \|A_{i_{k-1}}\|^2}{(A_{i_{k-1}}w^{k-1}-b_{i_{k-1}})[\|A_{i_{k-1}}\|^2\|A_{i_k}\|^2-\langle A_{i_{k-1}}, A_{i_k}\rangle^2]}.
\end{equation}
Consequently, (\ref{eq:4.4}) is of the form:
\begin{equation}\label{4new2}
x^{k+1}=P_{i_k}(x^k+\gamma_k^\prime(x^k-w^{k-1})).
\end{equation}
It is easy to see from (\ref{4new1}) and (\ref{4new2}) that constructing $x^{k+1}$ involves $x^0, x^1, \cdots, x^k$, hence Algorithm 4 is  a  multi-step inertial algorithm (see, e.g. \cite{Dong}).

(ii) Assume that $A$ is standardized. Set $\mu_k=\langle A_{i_{k-1}},A_{i_k}\rangle$ and $r^k=A_{i_k}x^k-b_{i_k}$. Then  (\ref{eq:4.2})-(\ref{eq:4.4}) of Algorithm 4 can be simply expressed as: for the current $x^k\,(k\geq 0)$, update $x^{k+1}$ by the formula
\begin{equation}\label{Al4.1-express}
x^{k+1}=x^k+\frac{r^k\mu_k}{1-\mu_k^2}A^\top_{i_{k-1}}-\frac{r^k}{1-\mu_k^2}A^\top_{i_k}.
\end{equation}
(\ref{Al4.1-express}) is very suitable for the implementation of Algorithm 4.

\end{remark}

\begin{lemma} \label{le:4.4}
The inertial parameter sequence $\{\gamma_k\}_{k=0}^\infty$ in \eqref{eq:4.2} is optimal in the sense that, for each $k\geq 1$,  $\gamma_k$ makes $\|x^{k+1}-x^*\|^2$ obtain its minimum value for arbitrary $x^*\in S$.
\end{lemma}
\begin{proof}
For each $k\geq 0$, it concludes from (\ref{eq:4.1}), (\ref{eq:4.4}), (\ref{eq:2.2}) and the fact  $A_{i_k}x^*=b_{i_k}$   that

\begin{equation}\label{eq:4.5}
\aligned
&\|x^{k+1}-x^*\|^2\\
=&\|P_{i_k}w^{k}-x^*\|^2\\
=&\left\|w^{k}-x^*-\frac{A_{i_k}w^{k}-b_{i_k}}{\|A_{i_k}\|^2}A_{i_k}^\top\right\|^2\\
=&\|w^{k}-x^*\|^2-2\left\langle w^{k}-x^*,\frac{A_{i_k}w^{k}-b_{i_k}}{\|A_{i_k}\|^2}A_{i_k}^\top\right\rangle +\frac{|A_{i_k}w^{k} -b_{i_k}|^2}{\|A_{i_k}\|^2}\\
=&\|w^{k}-x^*\|^2-\frac{|A_{i_k}w^{k}-b_{i_k}|^2}{\|A_{i_k}\|^2}.
\endaligned
\end{equation}
For any $k\geq 1$,  noting $A_{i_{k-1}}x^*=b_{i_{k-1}}$, we have from (\ref{add1}) that
\begin{equation}\label{new-4-2}
\langle x^k-x^*, A_{i_{k-1}}^\top\rangle =0.
\end{equation}
Consequently, we get
\begin{equation}\label{eq:4.6}
\|w^k-x^*\|^2=\|x^k-x^*+\gamma_kA_{i_{k-1}}^\top\|^2=\|x^k-x^*\|^2+\gamma_k^2\|A_{i_{k-1}}\|^2,\,\,\forall k\geq 1.
\end{equation}
By direct calculation, we get
\begin{equation}\label{eq:4.7}
\aligned
&|A_{i_k}w^k-b_{i_k}|^2\\
=&|A_{i_k}(x^{k}+\gamma_kA_{i_{k-1}}^\top)-b_{i_k}|^2\\
=&|A_{i_k}x^{k}-b_{i_k}|^2+\gamma_k^2\langle A_{i_{k-1}}, A_{i_k}\rangle^2+2\gamma_k\langle A_{i_{k-1}}, A_{i_k}\rangle(A_{i_k}x^{k}-b_{i_k}),\,\,\,\forall k\geq 1.
\endaligned
\end{equation}
Combining (\ref{eq:4.5}), (\ref{eq:4.6}) and (\ref{eq:4.7}), we obtain
\begin{equation}\label{eq:4.8}
\aligned
&\|x^{k+1}-x^*\|^2\\
=&\|x^{k}-x^*\|^2-\frac{|A_{i_k}x^{k}-b_{i_k}|^2}{\|A_{i_k}\|^2}+\gamma_k^2\left(\|A_{i_{k-1}}\|^2-\frac{\langle A_{i_{k-1}}, A_{i_k}\rangle^2}{\|A_{i_k}\|^2}\right)\\
&-2\gamma_k \frac{\langle A_{i_{k-1}}, A_{i_k}\rangle(A_{i_k}x^{k}-b_{i_k})}{\|A_{i_k}\|^2},\,\,\,\forall k\geq 1.
\endaligned
\end{equation}
By an argument similar to (iv) of Lemma \ref{le:3.3}, it is easy to see that $\gamma_k$ in (\ref{eq:4.2}) such that  $\|x^{k+1}-x^*\|^2$
reaches  its minimum value.
\end{proof}
\begin{lemma} \label{le:4.5}
Let $\{x^k\}_{k=0}^\infty$ be the  sequence generated by Algorithm {\rm4}. Then for any $k\geq 1$ and $x^* \in S$, the following statements hold:
\begin{enumerate}
\item [{\rm(i)}]

\begin{equation}\label{eq:4.10}
\|x^{k+1}-x^*\|^2=\|x^{k}-x^*\|^2-\frac{\displaystyle|A_{i_k}x^{k}-b_{i_k}|^2}{\displaystyle\|A_{i_k}\|^2}
\frac{\displaystyle1}{\displaystyle1-\left\langle \frac{\displaystyle A_{i_{k-1}}}{\displaystyle\|A_{i_{k-1}}\|}, \frac{\displaystyle A_{i_k}}{\displaystyle\|A_{i_k}\|}\right\rangle^2}.
\end{equation}

\item [{\rm(ii)}]

\begin{equation}\label{eq:4.12}
\|x^{k+1}-x^k\|^2=\frac{\displaystyle|A_{i_k}x^{k}-b_{i_k}|^2}{\displaystyle\|A_{i_k}\|^2}
\frac{\displaystyle1}{\displaystyle1-\left\langle \frac{\displaystyle A_{i_{k-1}}}{\displaystyle\|A_{i_{k-1}}\|}, \frac{\displaystyle A_{i_k}}{\displaystyle\|A_{i_k}\|}\right\rangle^2}.
\end{equation}

\end{enumerate}
\end{lemma}
\begin{proof}
(i)  Obviously,   (\ref{eq:4.10}) follows from substituting (\ref{eq:4.2}) into (\ref{eq:4.8}).

(ii) From (\ref{eq:4.3}) and (\ref{eq:4.4}), we have
\begin{equation*}
x^{k+1}-x^k=\gamma_k\left\{A_{i_{k-1}}^\top-\frac{\langle A_{i_{k-1}}, A_{i_k}\rangle}{\|A_{i_k}\|^2}A_{i_k}^\top\right\}-\frac{A_{i_k}x^{k}-b_{i_k}}{\|A_{i_k}\|^2}A_{i_k}^\top.
\end{equation*}
Noting that
$$\left\langle A_{i_{k-1}}^\top-\frac{\langle A_{i_{k-1}}, A_{i_k}\rangle}{\|A_{i_k}\|^2}A_{i_k}^\top, A_{i_k}^\top\right\rangle=0,$$
we have
\begin{equation}\label{eq:4.13}
\|x^{k+1}-x^k\|^2=\frac{|A_{i_k}x^{k}-b_{i_k}|^2}{\|A_{i_k}\|^2}+\gamma_k^2\left(\|A_{i_{k-1}}\|^2-\frac{\langle A_{i_{k-1}}, A_{i_k}\rangle^2}{\|A_{i_k}\|^2}\right)
\end{equation}
Hence, (\ref{eq:4.12}) yields  by substituting (\ref{eq:4.2}) into (\ref{eq:4.13}).
\end{proof}

\begin{theorem}
Choose $x^0\in {\rm span}\{A_1^\top, A_2^\top, \cdots, A_I^\top\}$ arbitrarily and let $\left\{x^k\right\}_{k=0}^{\infty}$ be a  sequence generated by  Algorithm {\rm4}. Then $\left\{x^k\right\}_{k=0}^{\infty}$ converges linearly in expectation  to the minimum norm solution $x^\dag$. Precisely, for all $k\geq 1$, there holds the following statement:

 \begin{equation} \label{e244}
\mathbb{E}\left[\|x^{k+1}-x^\dag\|^2\right]\leq \left(1-\frac{\lambda_{\min}(A^\top A)}{\alpha_F (A)(1-\delta^2(A))}\right)^k
\left(1-\frac{\lambda_{\min}(A^\top A)}{\|A\|_F^2}\right)\|x^{0}-x^\dag\|^2.
\end{equation}

    \end{theorem}
\begin{proof}
Note that $w^0=x^0$,  we have  from (\ref{eq:4.1}), (\ref{eq:4.5}) and (\ref{eq:2.2}) that
\begin{equation*}\label{eq:4.9}
\|x^{1}-x^\dag\|^2=\|x^{0}-x^\dag\|^2-\frac{|A_{i_{0}}x^{0}-b_{i_{0}}|^2}{\|A_{i_{0}}\|^2}.
\end{equation*}
Due to $x^0-x^\dag\in {\rm span}\{A_1^\top, A_2^\top, \cdots, A_I^\top\}$, this together with Lemma \ref{le:2.3} leads to

 \begin{equation} \label{e234}
 \aligned
 \mathbb{E}\left[\|x^{1}-x^\dag\|^2\right]&=\|x^{0}-x^\dag\|^2-\sum_{i_0\in \mathcal{I}}\frac{\|A_{i_{0}}\|^2}{\|A\|_F^2}\frac{|A_{i_{0}}x^{0}-b_{i_{0}}|^2}{\|A_{i_{0}}\|^2}\\
 &=\|x^{0}-x^\dag\|^2-\sum_{i_0\in \mathcal{I}}\frac{\|A_{i_{0}}\|^2}{\|A\|_F^2}\frac{|A_{i_{0}}(x^{0}-x^\dag)|^2}{\|A_{i_{0}}\|^2}\\
 &\leq \|x^{0}-x^\dag\|^2-\frac{1}{\|A\|_F^2}\|A(x^0-x^\dag)\|^2\\
 &\leq \left(1-\frac{\lambda_{\min}(A^\top A)}{\|A\|_F^2}\right)\|x^{0}-x^\dag\|^2.
\endaligned
\end{equation}
 For any $k\geq 1$, from (\ref{eq:4.10}), we have
\begin{equation} \label{new-e234}
\|x^{k+1}-x^\dag\|^2\leq\|x^{k}-x^\dag\|^2-\frac{\displaystyle|A_{i_k}x^{k}-b_{i_k}|^2}{\displaystyle\|A_{i_k}\|^2}
\frac{\displaystyle1}{\displaystyle1-\delta^2(A)}.
\end{equation}
From Step 1 of  Algorithm 4 and (\ref{new-e234}), we get
\begin{equation} \label{new-4-1}
 \aligned
&\mathbb{E}\left[\|x^{k+1}-x^\dag\|^2\,|\,x^k\right]\\
\leq &\|x^{k}-x^\dag\|^2-\frac{\displaystyle1}{\displaystyle1-\delta^2(A)}
\sum_{i_k\in\mathcal{I}-\{i_{k-1}\}}\frac{\|A_{i_k}\|^2}{\|A\|^2_F-\|A_{i_{k-1}}\|^2}\frac{\displaystyle|A_{i_k}x^{k}-b_{i_k}|^2}{\displaystyle\|A_{i_k}\|^2}\\
\leq &\|x^{k}-x^\dag\|^2-\frac{\displaystyle1}{\displaystyle\alpha_F(A)(1-\delta^2(A))}\sum_{i_k\in\mathcal{I}-\{i_{k-1}\}}|A_{i_k}x^{k}-b_{i_k}|^2.
\endaligned
\end{equation}
On the other hand,  from (\ref{new-4-2}), it yields $A_{i_{k-1}}x^k-b_{i_{k-1}}=0$. This together with Lemma \ref{le:2.3} leads to
\begin{equation} \label{new-4-3}
\aligned
&\sum_{i_k\in\mathcal{I}-\{i_{k-1}\}}|A_{i_k}x^{k}-b_{i_k}|^2\\
=&\sum_{i_k\in\mathcal{I}}|A_{i_k}x^{k}-b_{i_k}|^2=\sum_{i_k\in\mathcal{I}}|A_{i_k}(x^{k}-x^\dag)|^2
\geq\lambda_{\min}(A^\top A)\|x^k-x^\dag\|^2.
\endaligned
 \end{equation}
Combining (\ref{new-4-1}) and (\ref{new-4-3}), it yields
\begin{equation} \label{new-4-4}
\mathbb{E}\left[\|x^{k+1}-x^\dag\|^2\,|\,x^k\right]\leq \left(1-\frac{\lambda_{\min}(A^\top A)}{\alpha_F(A) (\displaystyle1-\delta^2(A))}\right)\|x^k-x^\dag\|^2.
\end{equation}
Hence, taking the full expectation at both sides of (\ref{new-4-4}), we obtain
\begin{equation} \label{new-4-4}
\mathbb{E}\left[\|x^{k+1}-x^\dag\|^2\right]\leq \left(1-\frac{\lambda_{\min}(A^\top A)}{\alpha_F(A) (\displaystyle1-\delta^2(A))}\right)\mathbb{E}\left[\|x^k-x^\dag\|^2\right].
\end{equation}
Using mathematical induction, (\ref{e244}) follows from (\ref{e234}) and (\ref{new-4-4}).
 \end{proof}

\begin{remark}
\rm 
Since each iteration of Algorithm 3 utilizes two rows of the matrix $A$, one needs to equate a single iteration of Algorithm 3 with two iterations of Algorithm 4 for fair comparison of the convergence rate. It is obvious that the estimate \eqref{e244} is better than the estimate \eqref{e24}  in this sense.

\end{remark}

\section{Numerical Results}
In this section, we perform several experiments to compare the convergence speed of the multi-step inertial randomized Kaczmarz algorithm (denoted as MIRK) with that of the two-subspace Kaczmarz method (denoted as TSK).
\vskip 1mm

We use ``IT" and ``CPU" to represent  the arithmetical averages of the required numbers of iteration steps required and  the elapsed CPU times taken to run the algorithms 50 times, respectively. The speed-up of MIRK against TSK is defined by
$$
\hbox{speed-up}=\frac{\hbox{CPU of TSK}}{\hbox{CPU of MIRK}}.
$$

To test methods, we construct various types of $I\times J$ matrices $A$   and set the entries of $A$ to be independent identically
distributed uniform random variables on some interval $[c,1]$. Changing the value of $c$ will appropriately change the coherence of  $A$.

In the implementation of the methods, we use the function $unifrnd$ to randomly generate a solution vector $x^*\in \mathbb{R}^J$,  and take the right-hand side $b\in \mathbb{R}^I$ as $Ax^*$. This ensures that the randomly generated  matrices constitute a solvable  system of equations $Ax=b$. For each matrix construction, two algorithms are run with the  initial vector $x^0=0$ and the same fixed matrix, and terminated once the relative solution error, defined by
$$
\hbox{RSE}=\frac{\|x^k-x^\dag\|^2}{\|x^\dag\|^2}
$$
at the current $x^k$, satisfies $\hbox{RSE}\leq 10^{-6}$. In addition, all experiments are carried out using MATLAB (version R2018a) on a personal computer with 2.30 GHz central processing unit (Intel(R) Core(TM) i5-6300HQ CPU), 12.00GB memory, and Windows operating system (Windows 10).
\vskip 1mm

\begin{table}[!h]
\caption{ IT and CPU of TSK and MIRK for $1000\times 3000$ matrices with different $c$. }
\begin{tabular}{|c|c|c|c|c|c|}
\hline
\multicolumn{2}{|c|}{$c$}&0.9&0.5&0.1&-0.4\\\hline
\multirow{2}*{TSK}&CPU&3.5000&3.3125&3.1719&2.2188\\\cline{2-6}
                   &IT& $2.7362\times 10^4$ & $2.6849\times 10^4$ & $2.6548\times 10^4$ & $2.3490\times 10^4$ \\\hline
\multirow{2}*{MIRK}&CPU&2.4844&2.4688&2.4063&2.2031\\\cline{2-6}
                   &IT& $3.7174\times 10^4$ & $3.7282\times 10^4$ & $3.6742\times 10^4$ & $3.3572\times 10^4$ \\\hline
\multicolumn{2}{|c|}{speed-up}&1.4088&1.3417&1.3182&1.0071\\\hline
\end{tabular}
\end{table}

\begin{table}[!h]
\caption{ IT and CPU of TSK and MIRK for $2000\times 1000$ matrices with different $c$. }
\begin{tabular}{|c|c|c|c|c|c|}
\hline
\multicolumn{2}{|c|}{$c$}&0.9&0.5&0.1&-0.2\\\hline
\multirow{2}*{TSK}&CPU&3.4531&3.3594&2.9375&2.7031\\\cline{2-6}
                   &IT& $5.0883\times 10^4$ & $4.7767\times 10^4$ & $4.5534\times 10^4$ & $6.6282\times 10^4$ \\\hline
\multirow{2}*{MIRK}&CPU&2.7969&2.7344&2.6563&2.6250\\\cline{2-6}
                   &IT& $6.8314\times 10^4$ & $6.7658\times 10^4$ & $6.7208\times 10^4$ & $7.9606\times 10^4$ \\\hline
\multicolumn{2}{|c|}{speed-up}&1.2346&1.2286&1.1059&1.0298\\\hline
\end{tabular}
\end{table}

\begin{figure}[!h]
\centering
\tiny{(a)}\includegraphics[scale=0.46]{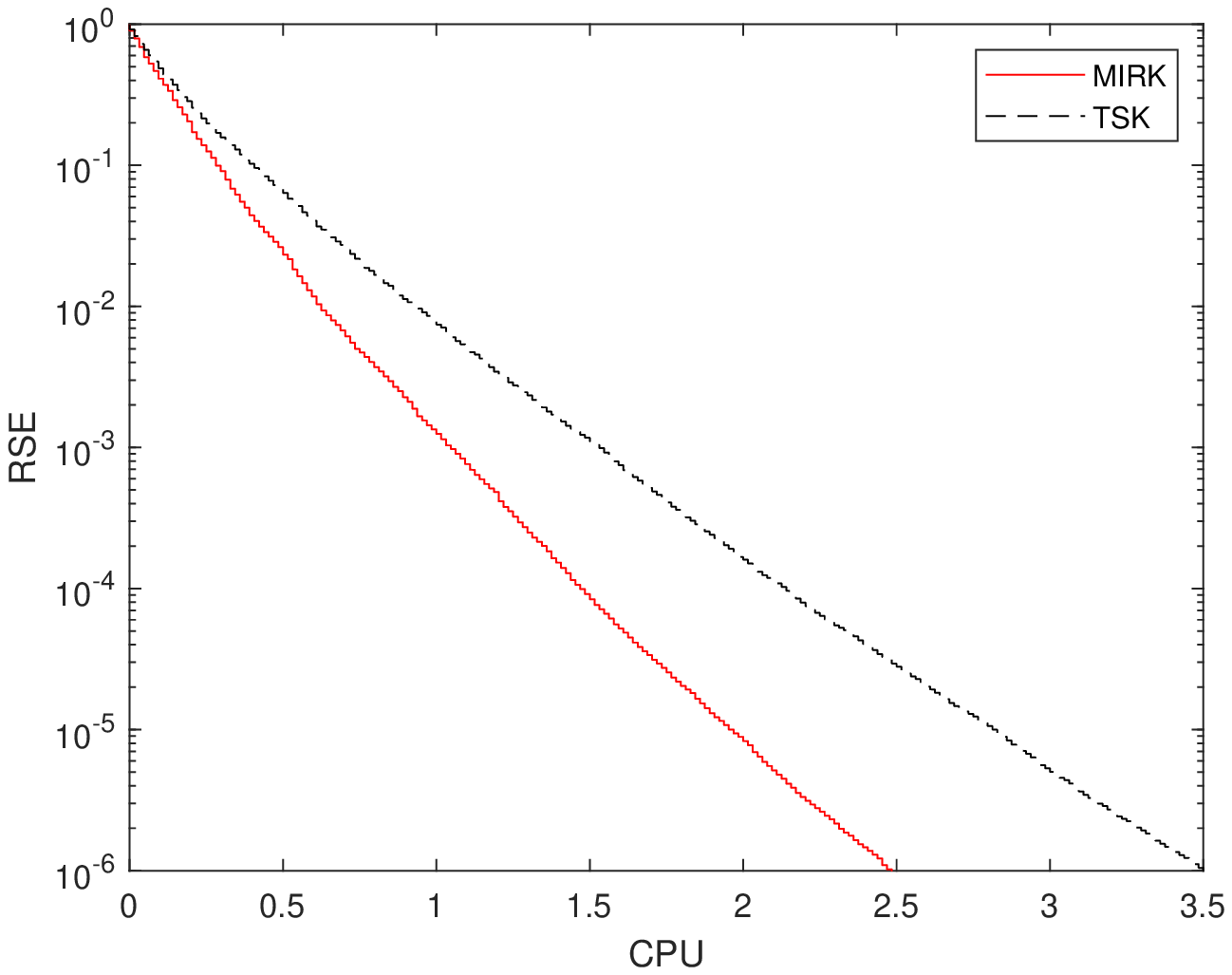}
\tiny{(b)}\includegraphics[scale=0.46]{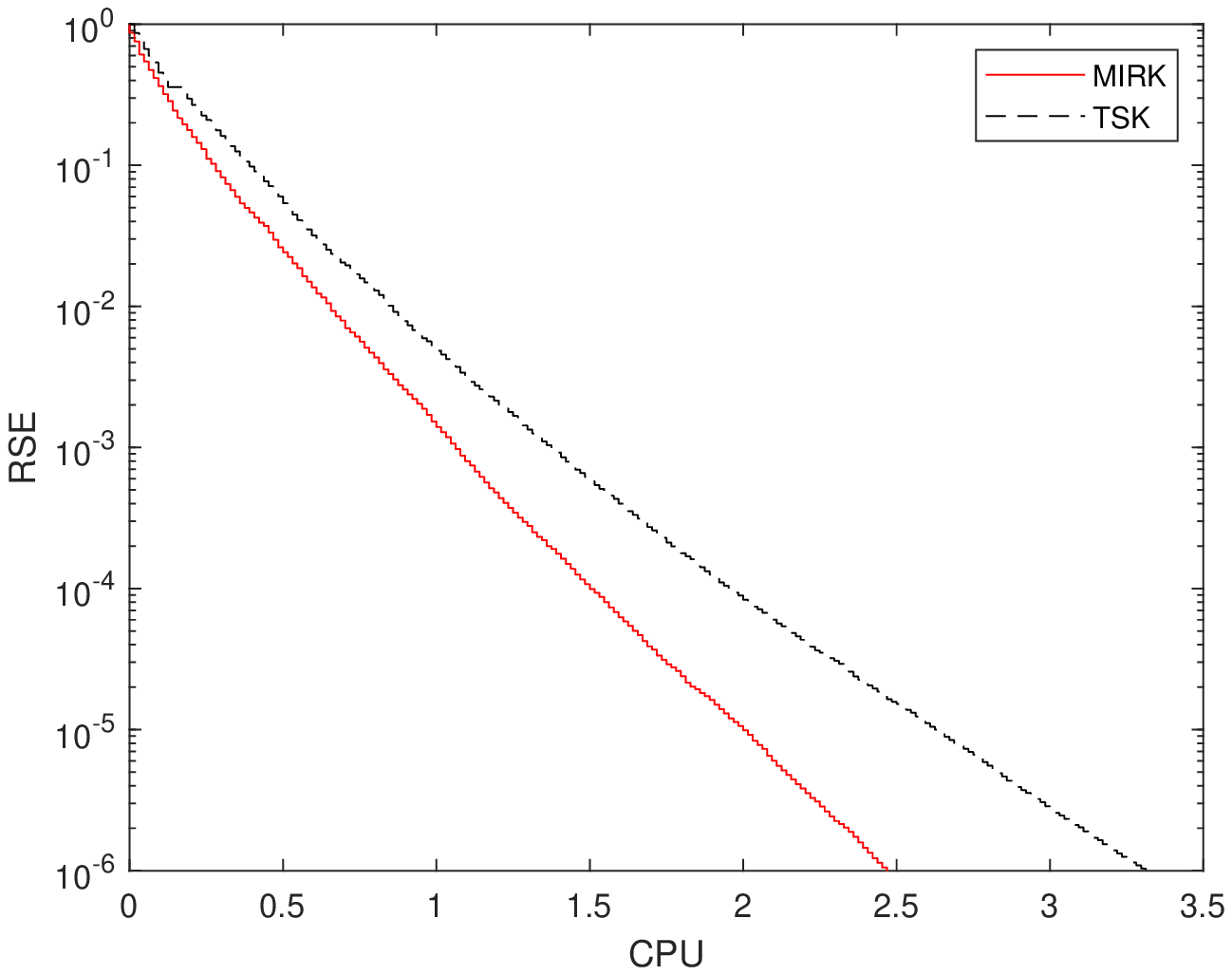}
\tiny{(c)}\includegraphics[scale=0.46]{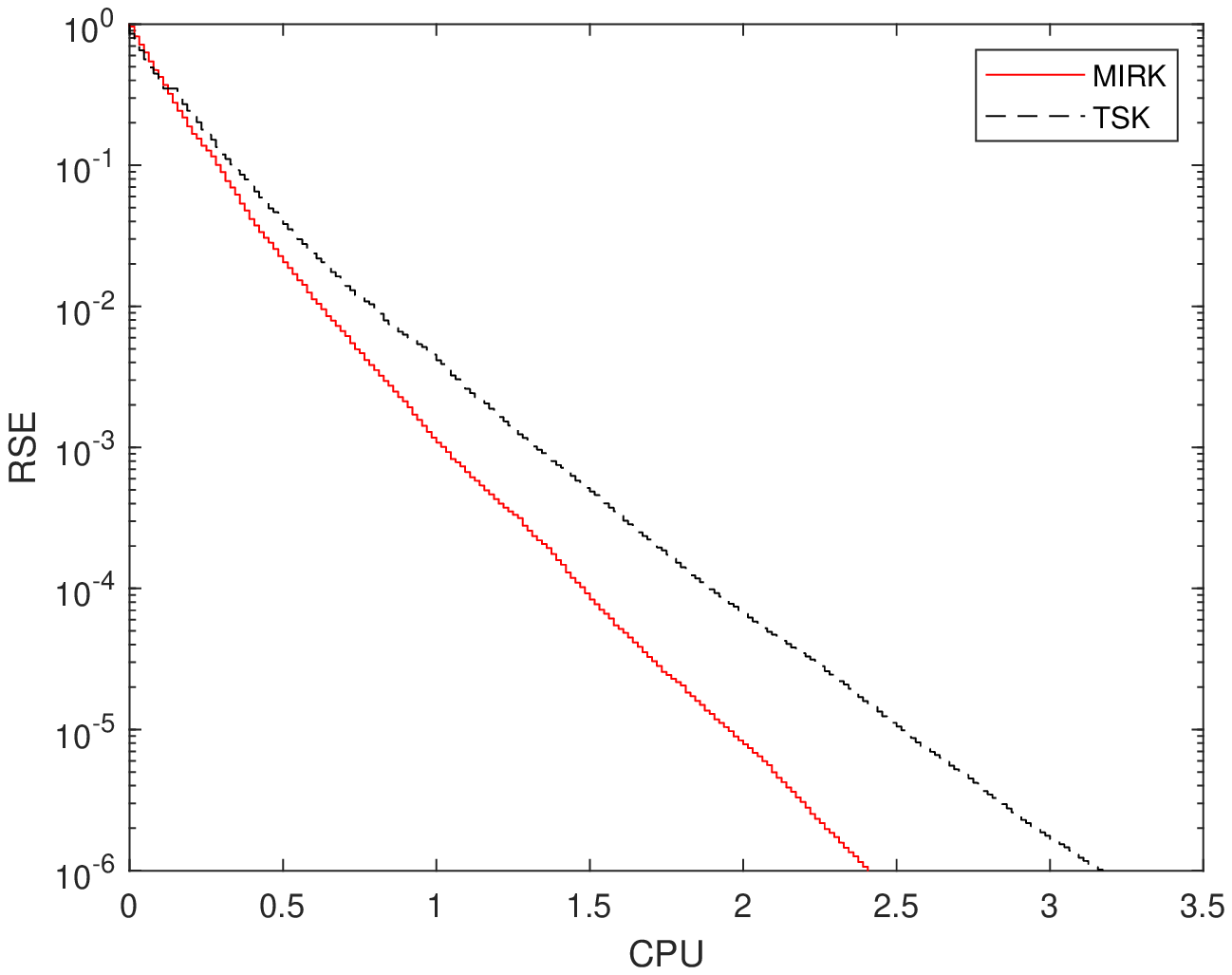}
\tiny{(b)}\includegraphics[scale=0.46]{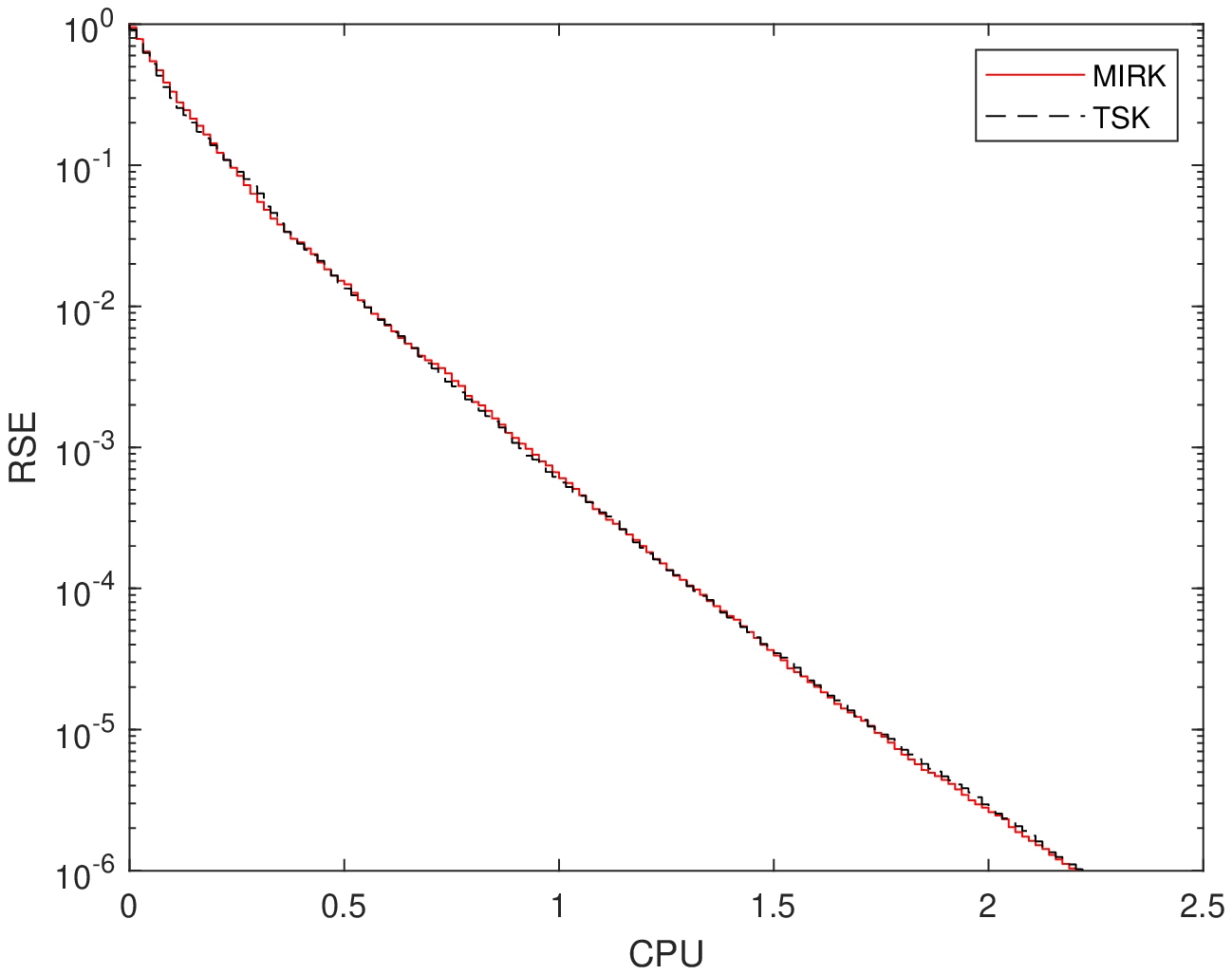}
\caption{RSE versus CPU with $1000\times 3000$ matrices for TSK and MIRK, when $c=0.9$(a), $c=0.5$(b), $c=0.1$(c) and $c=-0.4$(d).}\label{}
\end{figure}

\begin{figure}[!h]
\centering
\tiny{(a)}\includegraphics[scale=0.46]{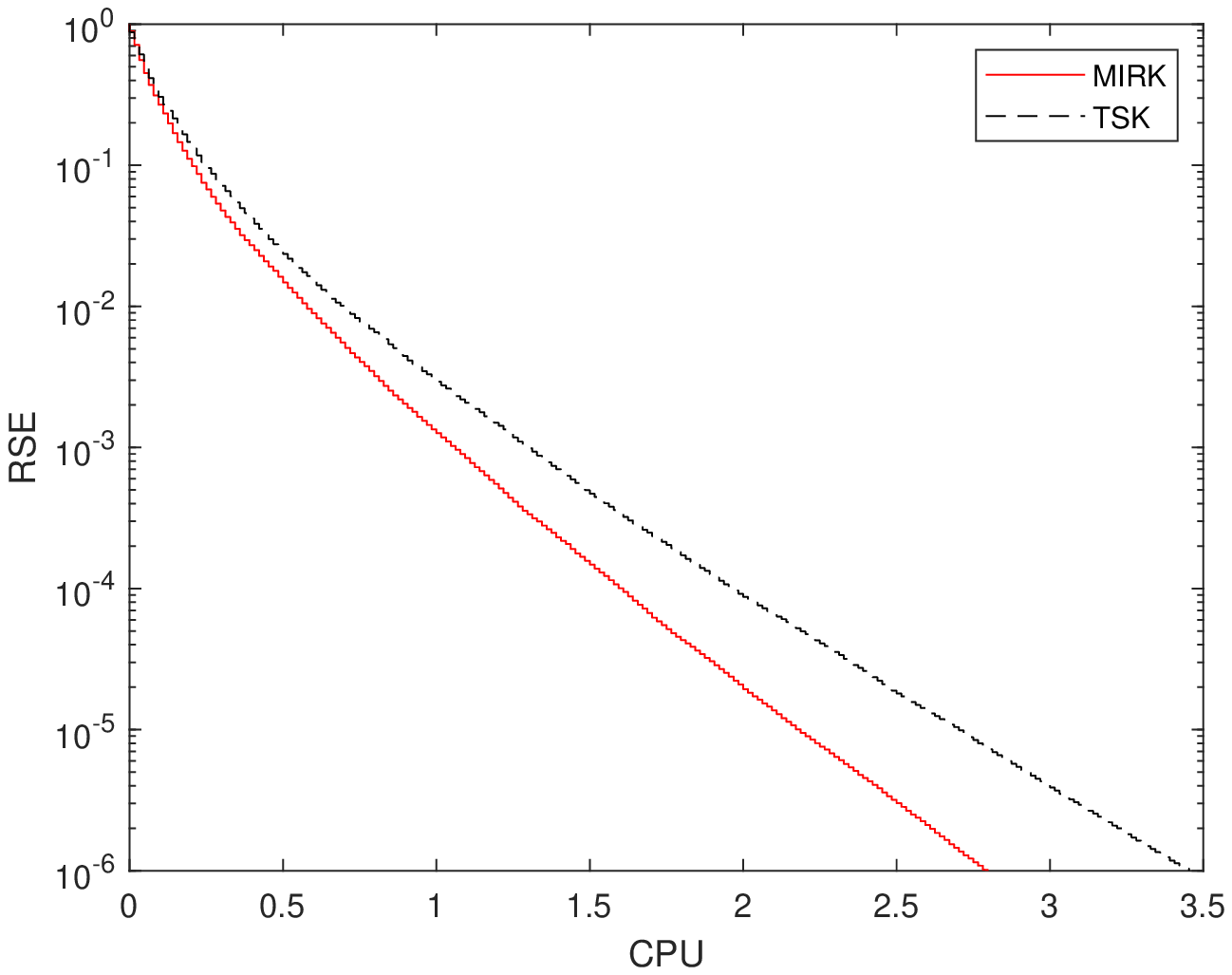}
\tiny{(b)}\includegraphics[scale=0.46]{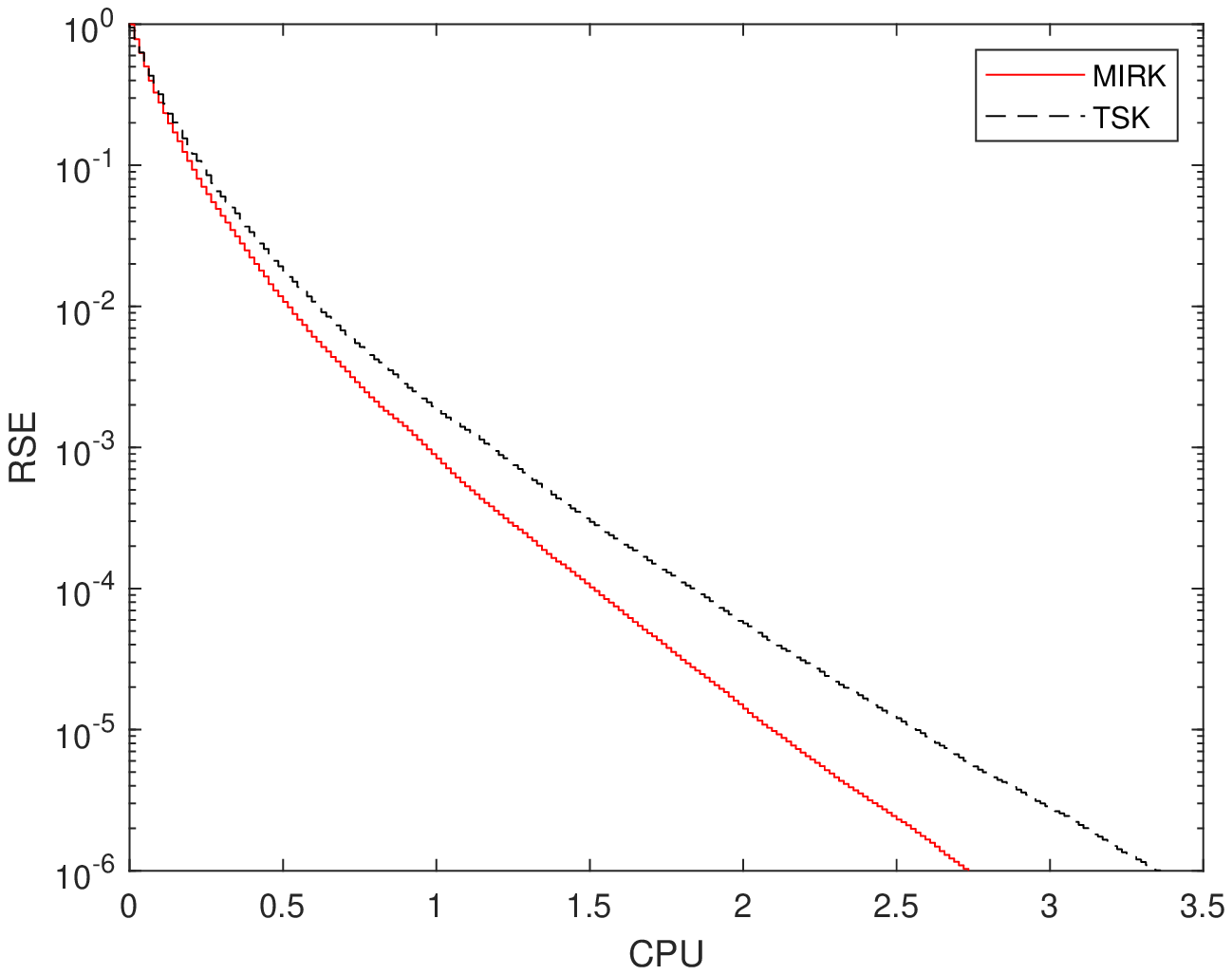}
\tiny{(c)}\includegraphics[scale=0.46]{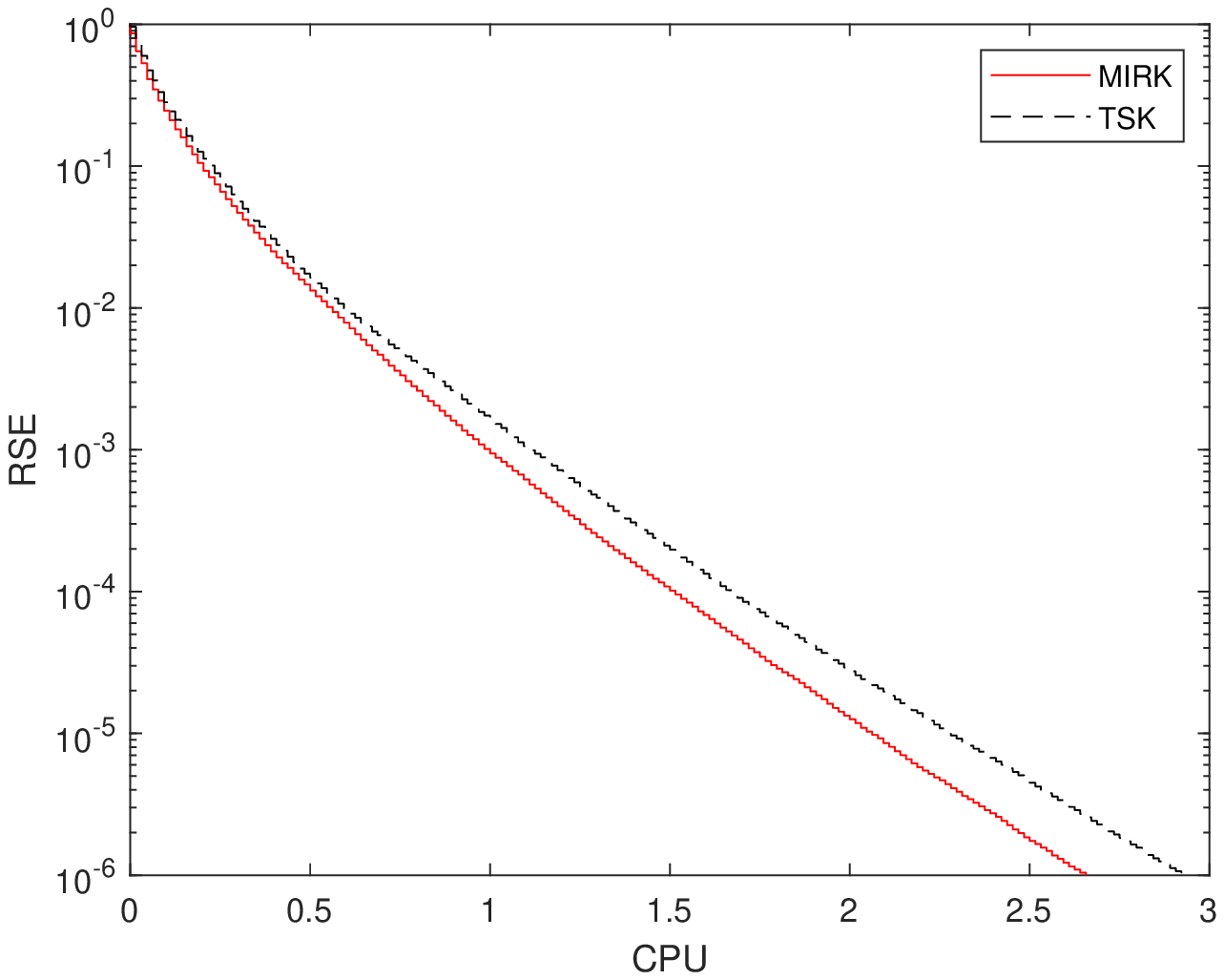}
\tiny{(b)}\includegraphics[scale=0.46]{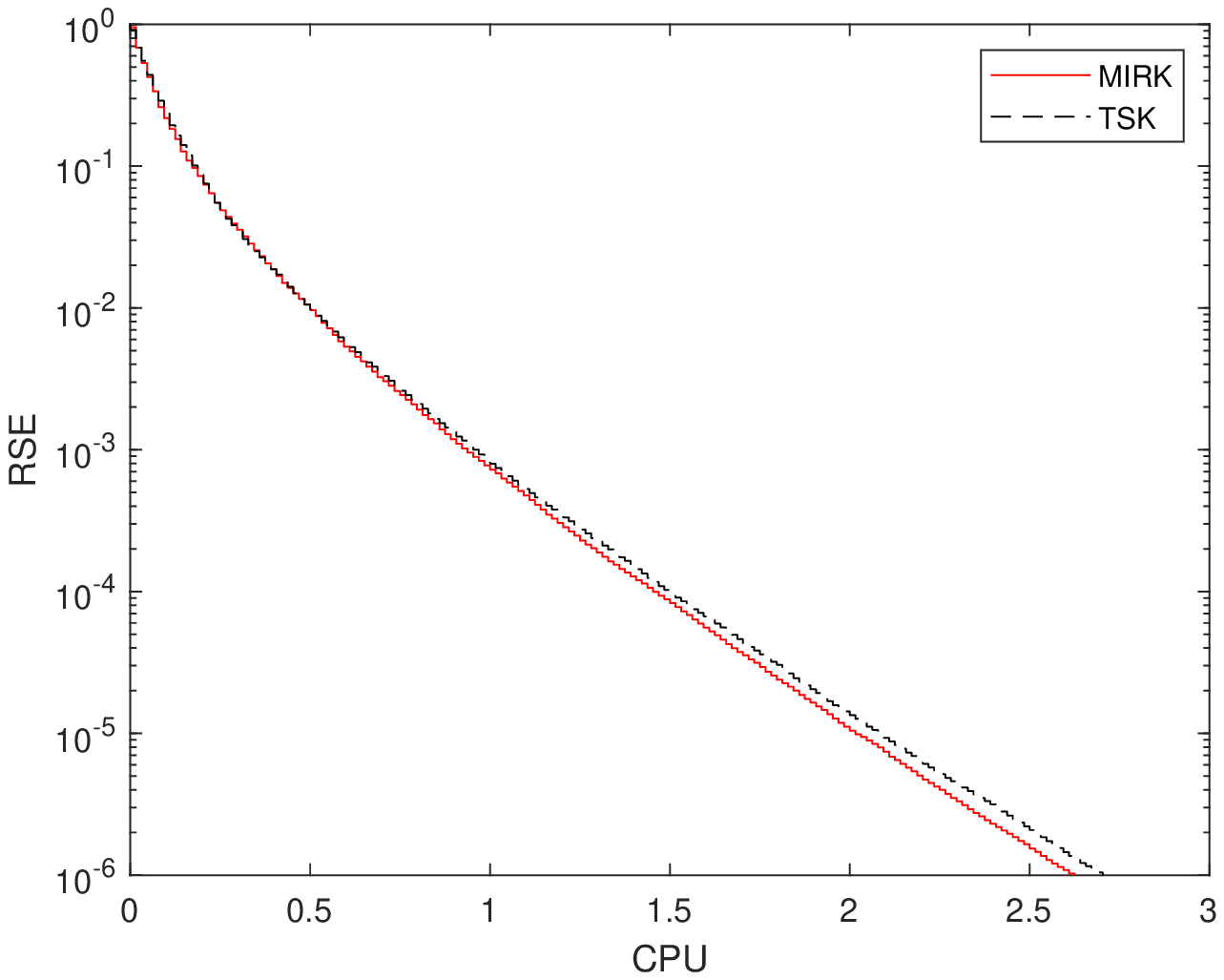}
\caption{RSE versus CPU with $2000\times 1000$ matrices for TSK and MIRK, when $c=0.9$(a), $c=0.5$(b), $c=0.1$(c) and $c=-0.2$(d).}\label{}
\end{figure}

\begin{figure}[!h]
\centering
\tiny{(a)}\includegraphics[scale=0.46]{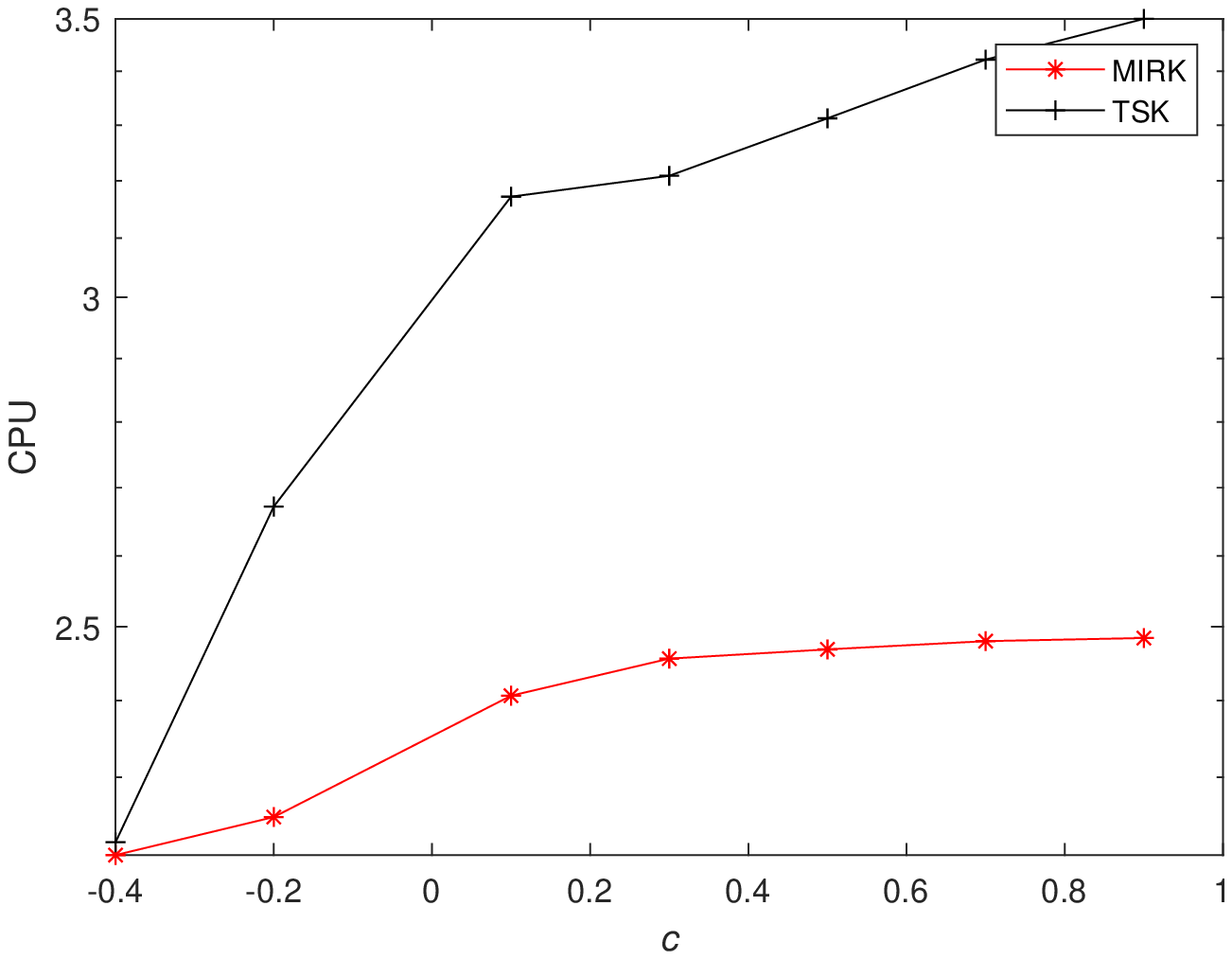}
\tiny{(b)}\includegraphics[scale=0.46]{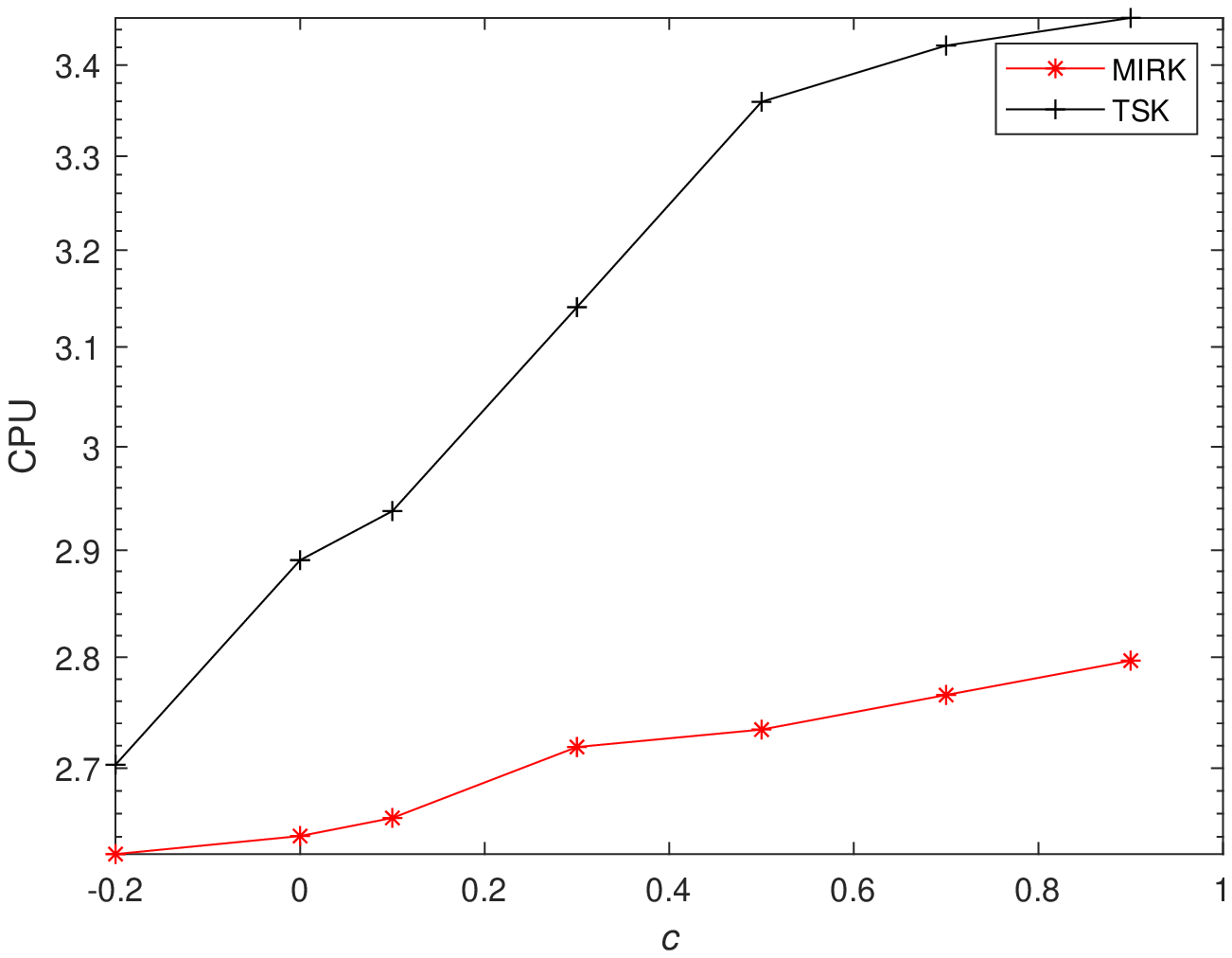}
\caption{CPU versus $c$ with $1000\times 3000$ matrices(a) and $2000\times 1000$ matrices(b) for TSK and MIRK.}\label{}
\end{figure}

\begin{figure}[!h]
\centering
\tiny{(a)}\includegraphics[scale=0.46]{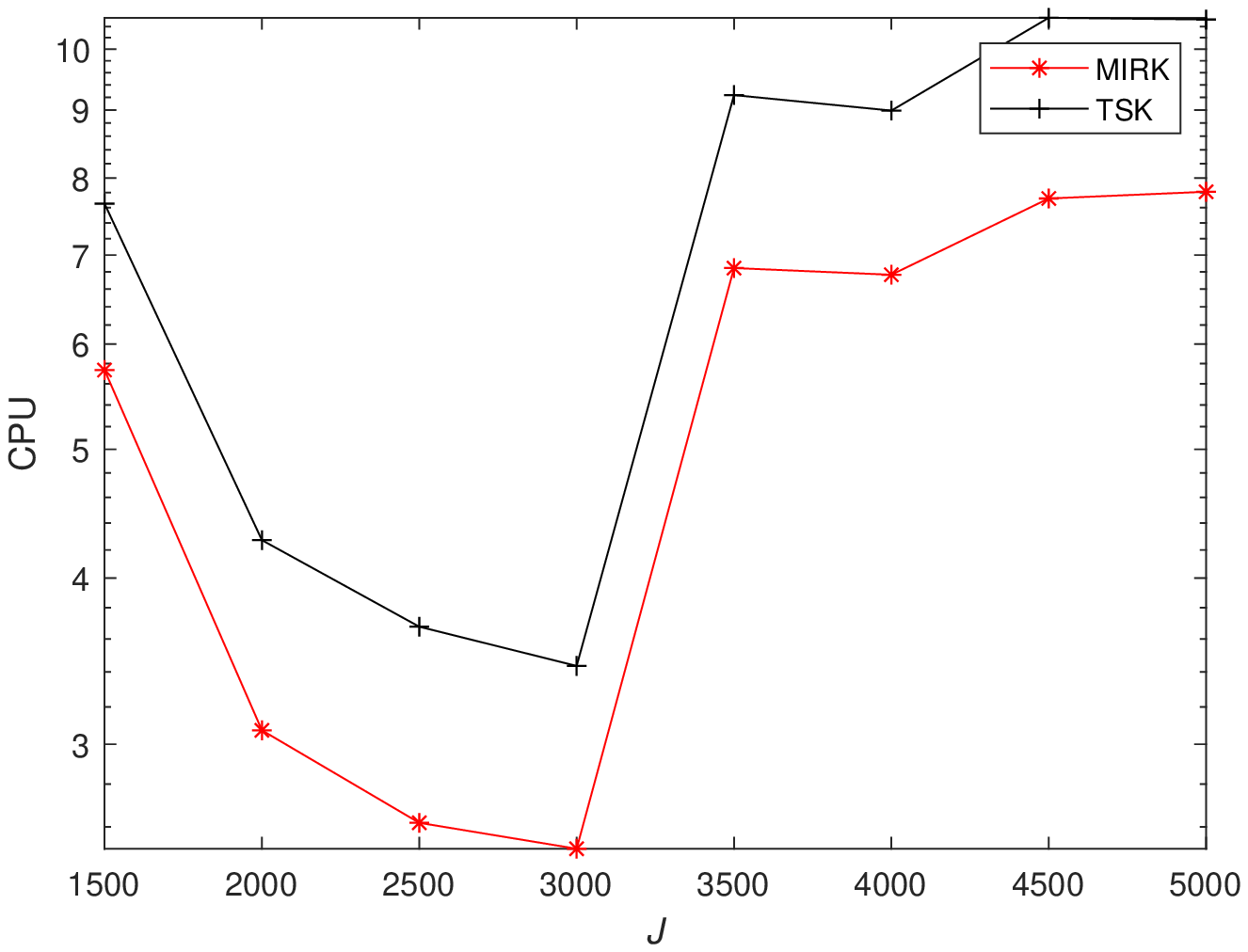}
\tiny{(b)}\includegraphics[scale=0.46]{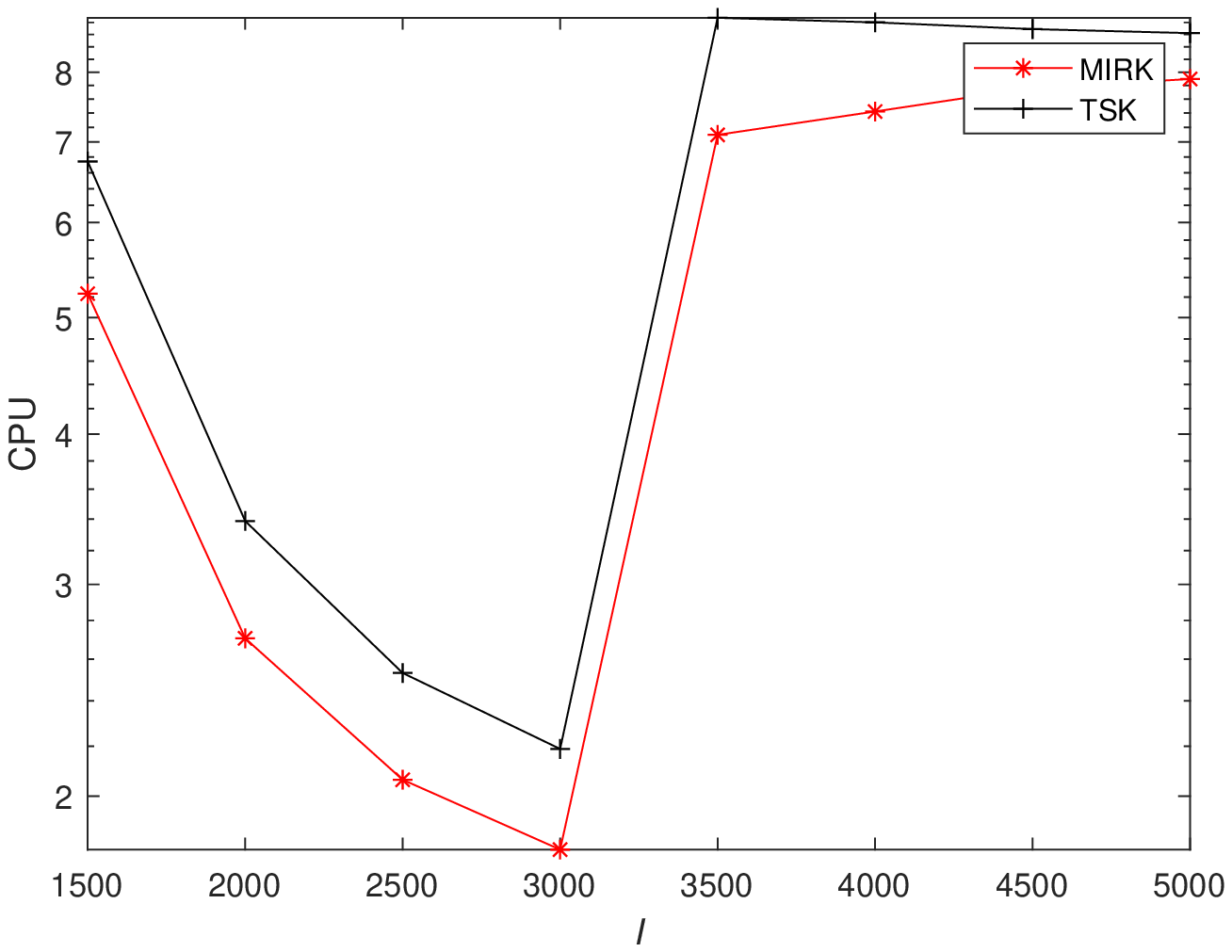}
\caption{ CPU versus $J$(a) and $I$(b) for TSK and MIRK when $I$ is fixed or $J$ is fixed.}\label{}
\end{figure}

We firstly take the  matrix sizes as $1000\times 3000$ and $2000\times 1000$, and  change the values of $c$. The numbers of iteration steps and the computing times for both MIRK and TSK methods are reported in Tables 1-2, from which we see that the MIRK algorithm always outperforms the TSK method for big $c$ in terms of CPU time, whether $A$ is fat or thin.   According to the data, the speed-up becomes larger as $c$ goes close to 1. When the  matrix is fat, the maximum speed-up can be 1.4088. When the matrix is thin, the maximum speed-up can be 1.2346. The data in Tables 1-2 is displayed directly in Figures 1-2 which  depict the curves of the relative solution error  versus CPU time  for the different rows and columns of  matrix $A$ and different values of $c$.  It can be concluded from the figures that when $c=0.9$, the two methods have the largest difference in the decay of the relative solution error and the MIRK algorithm is obviously better than the TSK method. The improvement effect of the MIRK algorithm over the TSK method is weakened  when $c$ successively decreases or even takes a negative value. It is observed that the drop curves almost coincide in Figure 1(d) ($c=-0.4$) and Figure 2(d) ($c=-0.2$), which means that the improvement effect disappears. These can also be seen in Figure 3,  which illustrats that the bigger $c$ is, the more obvious improvement effect of the MIRK algorithm over the TSK method.

\begin{table}[!h]
\caption{ IT and CPU of TSK and MIRK for $I$-by-$J$ matrices with $I=1000$ and different $J$. }
\begin{tabular}{|c|c|c|c|c|c|}
\hline
\multicolumn{2}{|c|}{$J$}&2000&3000&4000&5000\\\hline
\multirow{2}*{TSK}&CPU&4.2719&3.4356&8.9900&10.5266\\\cline{2-6}
                   &IT& $5.0309\times 10^4$ & $2.7114\times 10^4$ & $2.1190\times 10^4$ & $1.8685\times 10^4$ \\\hline
\multirow{2}*{MIRK}&CPU&3.0731&2.5031&6.7650&7.8109\\\cline{2-6}
                   &IT& $6.7512\times 10^4$ & $3.7046\times 10^4$ & $2.9676\times 10^4$ & $2.6438\times 10^4$ \\\hline
\multicolumn{2}{|c|}{speed-up}&1.3901&1.3725&1.3289&1.3477\\\hline
\end{tabular}
\end{table}

\begin{table}[!h]
\caption{ IT and CPU of TSK and MIRK for $I$-by-$J$ matrices with $J=1000$ and different $I$. }
\begin{tabular}{|c|c|c|c|c|c|}
\hline
\multicolumn{2}{|c|}{$I$}&2000&3000&4000&5000\\\hline
\multirow{2}*{TSK}&CPU&3.3872&2.1891&8.8016&8.6222\\\cline{2-6}
                   &IT& $5.0554\times 10^4$ & $2.7714\times 10^4$ & $2.1325\times 10^4$ & $1.8773\times 10^4$ \\\hline
\multirow{2}*{MIRK}&CPU&2.7056&1.8053&7.4206&7.8966\\\cline{2-6}
                   &IT& $6.7936\times 10^4$ & $3.7014\times 10^4$ & $2.9572\times 10^4$ & $2.6326\times 10^4$ \\\hline
\multicolumn{2}{|c|}{speed-up}&1.2519&1.2126&1.1861&1.0919\\\hline
\end{tabular}
\end{table}

\begin{figure}[!h]
\centering
\tiny{(a)}\includegraphics[scale=0.46]{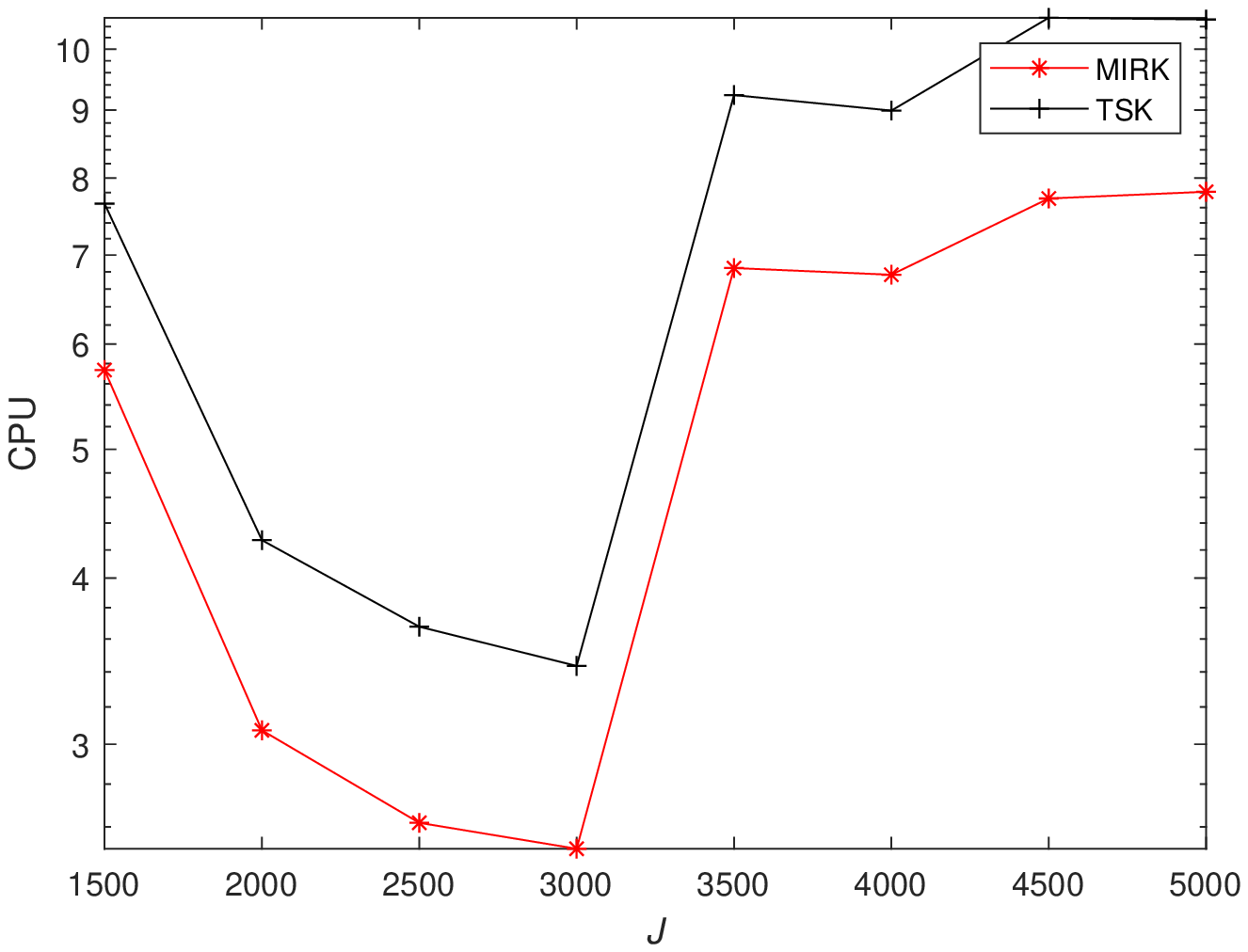}
\tiny{(b)}\includegraphics[scale=0.46]{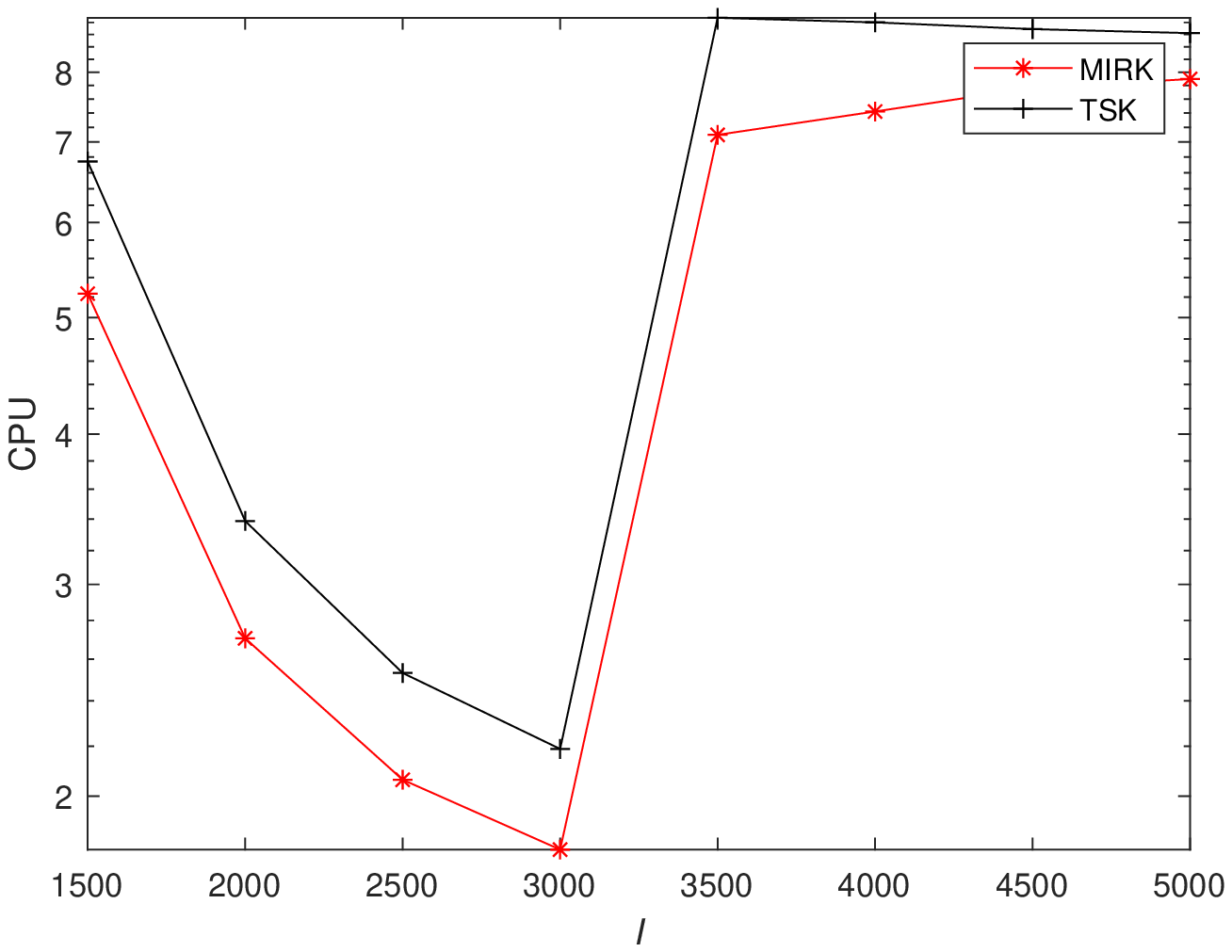}
\caption{ CPU versus $J$(a) and $I$(b) for TSK and MIRK when $I$ is fixed or $J$ is fixed.}\label{}
\end{figure}

We also fix $c=0.9$ and observe the performance of two algorithms for different row $I$ and and column $J$ of $A$.  We set $I=1000$ and change the values of $J$ in Table 3, which shows that in the fat case the speed-up can attain at most 1.3901 and least 1.3289.  We fix $J=1000$ and change the values of $I$ in Table 4 which illustrates that in thin case, the speed-up can be at most 1.2519 and at least 1.0919. It is concluded that the improvement effect is better for the fat case and decreases as $A$ becomes fatter or thinner. The above observations are visually plotted in Figure 4 which shows that  the  curves of the CPU time versus $J$ or $I$ of the MIRK algorithm are always lower than those of the TSK method no matter what the size of the matrix is.

{
\section{Some concluding remarks}

In this paper, we firstly give a new convergence rate of the two-subsapce Kaczmarz method by regarding it as an alternated inertial  randomized Kaczmarz algorithm. Secondly, we provide a multi-step inertial  randomized Kaczmarz algorithm which accelerates the alternated inertial  randomized Kaczmarz algorithm. The preliminary numerical examples are presented to support the theory results.

It is worth to investigate the combination of the inertial extrapolation with  variants of Kaczmarz method, for example the  greedy randomized Kaczmarz method \cite{BW} and the randomized block Kaczmarz method \cite{NT2014}.


\begin{thebibliography}{00}


\bibitem{Atto16} Attouch, H., Peypouquet, J., The rate of convergence of Nesterov's acclerated forward-backward method is actually faster than $1/k^2$, SIAM J. Optim. \textbf{26(3)}, (2016), 1824--1834.

{
\bibitem{BW}
Bai, Z.Z., Wu, W.T., On greedy randomized Kaczmarz method for solving
large sparse linear systems, SIAM J. Sci. Comput., \textbf{40}, (2018), A592-A606.
}

\bibitem{BC10} Bauschke, H.H.,  Combettes, P.L., Convex analysis and monotone
operator theory in Hilbert spaces. Springer 2010.

\bibitem{Be09} Beck, A., Teboulle,  M., A fast iterative shrinkage-thresholding algorithm for linear inverse
problems, SIAM J. Imaging Sci. \textbf{2(1)}(2009), 183--202.

\bibitem{1} Browne, J., DePierro, A.,  A row-action alternative to the EM algorithm for maximizing likelihoods in emission tomography, IEEE Trans. Med. Imag. \textbf{15}, (1996), 687--699.

\bibitem{BCL14} Byrne, C.L., Iterative Optimization in Inverse Problems.  CRC Press, Boca Raton (2014).

\bibitem{Censor81} Censor, Y., Row-action methods for huge and sparse systems and their applications. SIAM Review. \textbf{23}, (1981), 444--464.

\bibitem{Dong}
 Dong, Q.L.,  Huang, J., Li, X.H.,  Cho, Y.J., Rassias,  Th.M., MiKM: Multi-step
inertial Krasnosel'ski\v{\i}--Mann algorithm and its applications, J. Global Optim. \textbf{73(4)},
(2019), 801--824.

\bibitem{GR2015} Gower, R. M.,  Richtárik, P.,  Stochastic dual ascent for solving linear systems, arXiv: 1512.06890(2015), 28pages.

\bibitem{GR84} Goebel, K., Reich, S.,  Uniform convexity, hyperbolic geometry and
non-expansive mappings. M. Dekker (1984).

\bibitem{GOSB}
Goldstein, T.,   O'Donoghue, B.,  Setzer, S., Baraniuk, R., Fast alternating
direction optimization methods, SIAM J. Imaging Sc., \textbf{7(3)},
(2014), 1588--1623.

\bibitem{hyd} He, S.N., Yang, C.P.,  Duan, P.C., Realization of the hybrid
method for Mann iterations. Applied Mathematics and Computation. \textbf{217},
(2010), 4239--4247.

\bibitem{hm1993} Herman, G.T.,  Meyer, L., Algebraic reconstruction technique can be made computationlly efficient, IEEE T.  Med. Imaging, \textbf{12}, (1993), 600--609.

\bibitem{Iutzeler}
Iutzeler, F., Malick, J.,
 On the proximal gradient algorithm with alternated inertia, J. Optimiz. Theory App. \textbf{176}, (2018), 688--710.

\bibitem{KM1937}  Kaczmarz, S., Angen\"{a}herte Aufl\"{o}sung von Systemen linearer Gleichungen, Bull. Int. Acad. Polon. Sci. Lett. A, \textbf{35}, (1937), 355--357.

\bibitem{Liang}
Liang, J.W., Convergence Rates of First--Order Operator Splitting Methods. Optimization and Control [math.OC]. Normandie Universit\'{e}; GREYC CNRS UMR 6072, 2016. English.

\bibitem{MNR2015} Ma, A., Needell, D., Ramdas, A.,  Convergence properties of the randomized extended Gauss-Seidel and Kaczmarz methods, SIAM J. Matrix Anal. Appl., \textbf{36}, (2015), 1590--1604.

\bibitem{Mu-Peng} Mu, Z., Peng, Y., A note on the inertial proximal point method,
{Stat. Optim. Inf. Comput.} \textbf{3}, (2015), 241--248.


\bibitem{NT2014}
Needell, D., Tropp, J.A., Paved with good intentions: analysis of a randomized block Kaczmarz method, Linear  Algebra Appl. \textbf{441}, (2014), 199--221.


\bibitem{NW2013} Needell, D., Ward, R., Two-subspace projection method for coherent overdetermined systems, J. Fourier Anal. Appl., \textbf{19}, (2013), 256--269.

\bibitem{Nest83} Nesterov, Y.E.,  A method for solving the convex programming problem with convergence rate
$O(1/k^2)$, Dokl. Akad. Nauk SSSR, \textbf{269}, (1983), 543--547 (in Russian).


\bibitem{Po1964}  Polyak, B.T., Some methods of speeding up the convergence of iteration methods, U.S.S.R.
Comput. Math. Math. Phys. \textbf{4}, (1964), 1--17.

\bibitem{Shehu}
Shehu, Y., Iyiola, O.S., Projection methods with alternating inertial steps for variational inequalities: weak and linear convergence, Appl. Numer. Math.  \textbf{157}, (2020), 315--337.

\bibitem{SV}
Strohmer, T., Vershynin, R., A randomized Kaczmarz algorithm with exponential convergence, J. Fourier Anal. Appl., \textbf{15}, (2009),  262--278.


\end{thebibliography}
\end{document}